\documentclass{article}

\usepackage[utf8]{inputenc}
\usepackage[T1]{fontenc}    
\usepackage{url}        
\usepackage{geometry}
\usepackage{booktabs}   
\usepackage{multirow}
\usepackage{amsfonts}    
\usepackage{caption}
\usepackage{nicefrac}       
\usepackage{microtype}      
\usepackage{xcolor}         
\usepackage{subfigure}
\usepackage[inline]{enumitem}
\usepackage{amssymb,amsmath,amsthm, mathtools}
\usepackage{cases}
\usepackage{algorithm}
\usepackage{algorithmic}
\usepackage{graphicx}
\usepackage{dblfloatfix}
\usepackage{float}
\theoremstyle{plain}
\usepackage{makecell}
\usepackage{bm}
\usepackage{setspace}
\usepackage{thmtools}
\usepackage{thm-restate}
\usepackage{cases}
\usepackage{empheq}
\usepackage{titlesec}
\definecolor{dark-gray}{gray}{0.3}
\definecolor{dkgray}{rgb}{.4,.4,.4}
\definecolor{dkblue}{rgb}{0,0,.5}
\definecolor{medblue}{rgb}{0,0,.75}
\definecolor{rust}{rgb}{0.5,0.1,0.1}
\definecolor{darkblue}{rgb}{0,0.08,0.45}
\usepackage[colorlinks]{hyperref}
\usepackage{biblatex}
\hypersetup{urlcolor=rust}
\hypersetup{citecolor=darkblue}
\hypersetup{linkcolor=blue}
\newcommand{\norm}[1]{\left\lVert \, #1 \, \right\rVert}

\DeclareMathOperator*{\argmin}{arg\,min}

\numberwithin{theorem}{section}
\newtheorem{lemma}{Lemma}
\numberwithin{lemma}{section}

\numberwithin{remark}{section}

\numberwithin{corollary}{section}

\numberwithin{definition}{section}

\numberwithin{assumption}{section}

\declaretheorem[name=Theorem,numberwithin=section]{thm-rest}
\declaretheorem[name=Lemma,numberwithin=section]{lem-rest}
\declaretheorem[name=Corollary,numberwithin=section]{cor-rest}
\declaretheorem[name=Definition,numberwithin=section]{def-rest}
\declaretheorem[name=Proposition,numberwithin=section]{prop-rest}
\declaretheorem[name=Assumption,numberwithin=section]{ass-rest}

\titleformat{\section}
  {\normalfont\large\bfseries}{\thesection}{1em}{}
\titleformat{\subsection}
  {\normalfont\normalsize\bfseries}{\thesubsection}{1em}{}

\makeatletter
\def\@maketitle{%
  \newpage
  \null
  \vskip 2em%
  \begin{center}%
  \let \footnote \thanks
    {\Large \@title \par}
    \vskip 1.5em%
    {\normalsize
      \lineskip .5em%
      \begin{tabular}[t]{c}%
        \@author
      \end{tabular}\par}%
    \vskip 1em%
    {\normalsize \@date}%
  \end{center}%
  \par
  \vskip 1.5em}
\makeatother

\allowdisplaybreaks
\def\nn{{ \nonumber }}
\DeclareMathAlphabet{\mathcalorigin}{OMS}{cmsy}{m}{n}

\usepackage[colorinlistoftodos,prependcaption,backgroundcolor=black!5!white,bordercolor=red]{todonotes}

\title{\vspace{-2.0cm}  A Tuning-Free Primal-Dual Splitting Algorithm for Large-Scale Semidefinite Programming}
\date{}
\author{
    Yinjun Wang\thanks{Independent Researcher, yinjunwang2001@gmail.com} \and
    Haixiang Lan\thanks{Columbia University, hl3725@columbia.edu} \and
    Yinyu Ye\thanks{Stanford University, yinyu-ye@stanford.edu}
}
\begin{document}
\small
\maketitle
\begin{abstract}
This paper proposes and analyzes a tuning-free variant of Primal-Dual Hybrid Gradient (PDHG), and investigates its effectiveness for solving large-scale semidefinite programming (SDP). The core idea is based on the combination of two seemingly unrelated results: (1) the equivalence of PDHG and Douglas-Rachford splitting (DRS) \cite{o2020equivalence}; (2) the asymptotic convergence of \emph{non-stationary} DRS \cite{lorenz2019non}. This combination provides a unified approach to analyze the convergence of generic adaptive PDHG, including the proposed tuning-free algorithm and various existing ones. Numerical experiments are conducted to show the performance of our algorithm, highlighting its superior convergence speed and robustness in the context of SDP. 
\end{abstract}

\section{Introduction}
\label{sec:intro}
Semidefinite programming (SDP) is an effective framework for convex optimization, incorporating a wide range of important classes such as linear programming (LP), convex quadratic programming, and second-order cone programming. Aside from its recognized applications in mathematical modeling and constructing convex relaxations of NP-hard problems, SDP is notably effective in resolving various machine learning problems. They include but are not limited to maximum variance unfolding \cite{kulis2007fast}, sparse principal component analysis \cite{d2004direct}, matrix completion \cite{recht2010guaranteed}, graphical model inference \cite{erdogdu2017inference}, community detection \cite{bandeira2016low}, $k$-mean clustering \cite{peng2007approximating}, neural network verification \cite{albarghouthi2021introduction}, etc.

The wide adoption of SDP is facilitated by efficient and robust algorithms. Interior-point methods (IPMs), which can provably solve SDP to arbitrary precision in polynomial time, serve as foundational approaches in traditional SDP solvers \cite{gurobi}. However, their computational costs are typically high due to the procedure of adopting Newton's method to approximately solve a Karush-Kuhn-Tucker (KKT) system in each iteration. Consequently, first order methods, such as Alternating Directions Method of Multipliers (ADMM) \cite{boyd2011distributed} and Primal-Dual Hybrid Gradient (PDHG) \cite{chambolle2011first, condat2013primal}, have much lower complexity, and thus, have gained much popularity in recent years. Each iteration of PDHG is free from solving a linear system, in contrast to ADMM, which requires solving a linear system in every iteration. However, while ADMM has been the method of choice for several popular SDP solvers \cite{wen2010alternating,zheng2020chordal, garstka2021cosmo, o2016conic, sun2020sdpnal+,zhao2010newton}, PDHG for SDP has gained few attentions \cite{souto2022exploiting}. 

To effectively implement PDHG in practice, a key challenge is the tuning of its primal and dual stepsizes. The values of these two stepsizes significantly influence its convergence speed. Relying on manual tuning is not only computationally intensive -- requiring tests across various values -- but also lacks robustness, as the stepsizes that lead to fast convergence in one scenario may not be a good choice for another. Several stepsize-adjusting strategies have been proposed to address this challenge. The primal-dual balancing technique, proposed by \cite{goldstein2013adaptive, goldstein2015adaptive}, aims to dynamically adjust the ratio between the primal and dual stepsizes, keeping the primal and dual residuals approximately equal. The idea of balancing primal and dual residuals has been further extended by \cite{yokota2017efficient,zdun2021fast} to take into account the local variation of gradient directions. Moreover, a linesearch technique has been proposed in \cite{malitsky2018first}, adjusting stepsizes by repeated extrapolation and backtracking. Recently, \cite{applegate2021practical} focused on PDHG for LP and proposed a stepsize-adjusting heuristics by combining the ideas of both balancing and linesearch. Although these stepsize-adjusting strategies \cite{goldstein2013adaptive, goldstein2015adaptive,yokota2017efficient,zdun2021fast} can speed up convergence, they introduce some additional hyperparameters. These new hyperparameters still need to be carefully tuned, and as verified in our numerical experiments, manually tuning them suffers from issues similar to those encountered when tuning the original stepsizes.

The goal of this paper is to develop a variant of PDHG for SDP that is totally tuning-free. We first consider a generic PDHG variant (Algorithm~\ref{algm:pdhg_sdp}) that allows to dynamically adjust its stepsizes in each iteration, and study its convergence (Theorem~\ref{theorem_pdhg4sdp}). Motivated by the connection \cite{o2020equivalence} between PDHG and Douglas-Rachford splitting (DRS) \cite{douglas1956numerical, lions1979splitting}, we treat Algorithm~\ref{algm:pdhg_sdp} as DRS solving a specific class of monotone inclusion problem, and employ the convergence results from \cite{lorenz2019non} to prove its asymptotic convergence. Note that the class of monotone inclusion problem is only assumed to be maximally monotone. Thus, the asymptotic convergence is a fairly standard result, since no inequality from convexity can be exploited to prove a $\mathcal{O}(\frac{1}{k})$ convergence rate, where $k$ is the number of iteration. It is worth noting that although the theoretical results are presented in the context of SDP, they can be easily extended to the setting of $f(x) + g(Ax)$, where $f$ and $g$ are convex, closed, proper, and potentially nonsmooth. As a by-product, we show that various adaptive variants, such as balancing primal-dual residuals \cite{goldstein2013adaptive,goldstein2015adaptive} and aligning local variation of gradient directions \cite{yokota2017efficient,zdun2021fast} are special cases of Algorithm~\ref{algm:pdhg_sdp}, and prove their convergence in an unified manner (Theorem~\ref{theorem_pdhg4sdp_balancing}). For the tuning-free algorithm develpoment, the stepsize-adjusting strategy for DRS proposed in \cite{lorenz2019non} enables us to develop a tuning-free variant of PDHG  (Algorithm~\ref{algm:pdhg_sdp_drs}), which is also a special case of Algorithm~\ref{algm:pdhg_sdp} (Theorem~\ref{theorem_pdhg4sdp_drs}). Finally, we conduct numerical experiments and show that Algorithm~\ref{algm:pdhg_sdp_drs} exhibits compelling performances on both convergence speed and robustness of incorporating low-rank rounding. 

\section{Algorithms and Convergence}
\label{sec:algss-and-conv} 
\subsection{Notation}
Define the resolvent of an operator $A$ to be $J_{A} \coloneqq (\mathbb{I}+A)^{-1}$. Note that if $A = \partial f$, then $J_{\partial f} = \text{Prox}_{f}$. Given $A,B\in \mathbb{R}^{n \times n}$, define operation $\langle A, B\rangle \coloneqq \sum_{i=1}^{n}\sum_{j=1}^{n} A_{ij}B_{ij}$. Define a linear mapping $\mathcal{A} \colon \mathbb{S}^{n\times n} \to \mathbb{R}^m$ as $\mathcal{A}(X) = 
\begin{pmatrix}
\langle A_1, X \rangle \\
\langle A_2, X \rangle \\
\vdots \\
\langle A_m, X \rangle \\
\end{pmatrix}$, where $A_1, A_2,\dots, A_m \in \mathbb{S}^{n\times n}$ are $m$ symmetric matrices. Its adjoint mapping $\mathcal{A}^T \colon \mathbb{R}^m \to \mathbb{S}^{n\times n}$ is defined as $\mathcal{A}^T(y)=\sum_{i=1}^{m}y_i A_i$. SDP aims to minimize a linear objective function of $X \in \mathbb{R}^{n\times n}$ subject to $m$ linear equality constraints, and is defined as
\begin{equation}
\label{prob:primal_sdp}
\begin{aligned}
\min_{X \in \mathbb{R}^{n\times n}} \quad & \langle C, X \rangle\\
\textrm{s.t.} \quad & \mathcal{A}(X) = b \\
  & X \in  \mathbb{S}_{+}^{n\times n}   
\end{aligned}
\end{equation} 
Given $\mathbb{I}_{\mathbb{S}_{+}^{n\times n}}\left(X\right) = 0$ if $X\in \mathbb{S}_{+}^{n\times n}$ and $= \infty$ otherwise, and $\mathbb{I}_{=b}\left(y\right) = 0$ if $y = b$ and $= \infty$ otherwise, \eqref{prob:primal_sdp} can be rewritten as the following nonsmooth problem:
\begin{align}
\label{prob:primal_sdp_nonsmooth}
\min_{X} \quad  \underbrace{\langle C, X \rangle + \mathbb{I}_{\mathbb{S}_{+}^{n\times n}} (X)}_{f\left(\cdot\right)} +  \underbrace{\mathbb{I}_{=b} \left(\mathcal{A}(X)\right)}_{g\left(\cdot\right)}
\end{align}

\subsection{Tuning-Free PDHG for SDP}
We consider a generic variant of  PDHG for SDP (Algorithm~\ref{algm:pdhg_sdp}) that allows to dynamically adjust its stepsizes at each iteration. When $\theta_k$ is fixed to be $1$ for $\forall k$, it is reduced to
\begin{align*}
    & X^{k+1} \leftarrow \text{Proj}_{\mathbb{S}_{+}^{n\times n}}\Big(X^{k} - \alpha_{k} \big(\mathcal{A}^T(y^k)+C\big)\Big) \\
    & y^{k+1} \leftarrow y^k + \beta_k \mathcal{A}\left(2X^{k+1}-X^{k}\right) - \beta_k b,
\end{align*}
which is first presented in \cite{souto2022exploiting} and can be seen as PDHG solving \eqref{prob:primal_sdp_nonsmooth}. We develop Theorem~\ref{theorem_pdhg4sdp} to guide Algorithm~\ref{algm:pdhg_sdp} to adjust its stepsizes at each iteration. The idea comes from the connection \cite{o2020equivalence} between PDHG and DRS. By introducing a new variable $\hat{X}$, this connection allows us to treat PDHG as DRS 
\begin{align*}
    \begin{pmatrix}
        X^{k+1} \\
        \hat{X}^{k+1}
    \end{pmatrix} = J_{\alpha A}\left(2 J_{\alpha B}\left(\begin{pmatrix}
        X^{k} \\
        \hat{X}^{k}
    \end{pmatrix}\right)  - \begin{pmatrix}
        X^{k} \\
        \hat{X}^{k}
    \end{pmatrix}\right) + \begin{pmatrix}
        X^{k} \\
        \hat{X}^{k}
    \end{pmatrix} - J_{\alpha B}\left(\begin{pmatrix}
        X^{k} \\
        \hat{X}^{k}
    \end{pmatrix}\right)
\end{align*}
solving the following monotone inclusion problem:
\begin{align*}
    \text{Find}\ \begin{pmatrix}
        X \\
        \hat{X}
    \end{pmatrix}\quad \text{s.t.}\quad  0 \in  \underbrace{\begin{pmatrix} C \\
    0 \end{pmatrix} + \begin{pmatrix}\partial_{X}\mathbb{I}_{\mathbb{S}_{+}^{n\times n}} (X) \\
    0
    \end{pmatrix} + 
    \begin{pmatrix} 0 \\
    \partial_{\hat{X}}\mathbb{I}_{=0} (\hat{X})
    \end{pmatrix}}_{B = \partial f(X,\hat{X})} + \underbrace{\begin{pmatrix}  \partial_{X}\mathbb{I}_{=b} (\mathcal{A}(X) + \mathcal{T}(\hat{X})) \\ 
    \partial_{\hat{X}}\mathbb{I}_{=b} (\mathcal{A}(X) + \mathcal{T}(\hat{X}))
    \end{pmatrix}}_{A = \partial g(X,\hat{X})},
\end{align*}
where $f(X,\hat{X}) = \langle C, X \rangle + \mathbb{I}_{\mathbb{S}_{+}^{n\times n}} (X) + \mathbb{I}_{=0} (\hat{X})$, $g(X,\hat{X}) = \mathbb{I}_{=b} \left(\mathcal{A}(X) + \mathcal{T}(\hat{X})\right)$, and $\mathcal{T}\colon \mathbb{S}^{n\times n} \to \mathbb{R}^m$ is a certain linear operator that will be specified in the proof. We then can leverage the convergence result of \emph{non-stationary} DRS (Theorem 3.2. in \cite{lorenz2019non}) to develop the convergence theory for Algorithm~\ref{algm:pdhg_sdp}. The details of proof are presented in Appendix.

\begin{algorithm}[H]
\captionsetup{font=small}
\small
\begin{algorithmic}
\STATE \textbf{Input:} $X^0, y^0$
\FOR{$k = 0, 1, 2, \dots $}
\STATE $\alpha_k, \theta_k, \beta_k$ are adjusted following Theorem~\ref{theorem_pdhg4sdp}
\STATE $X^{k+1} \leftarrow \text{Proj}_{\mathbb{S}_{+}^{n\times n}}\Big(X^{k} - \alpha_{k} \big(\mathcal{A}^T(y^k)+C\big)\Big)$
\STATE $y^{k+1} \leftarrow y^k + \beta_k \mathcal{A}(X^{k+1}+\theta_k(X^{k+1}-X^{k})) - \beta_k b$
\ENDFOR
\end{algorithmic}
\caption{\small Generic Adaptive PDHG for SDP} 
\label{algm:pdhg_sdp}
\end{algorithm}

\begin{restatable}{thm-rest}{ThPDHGforSDP}
\label{theorem_pdhg4sdp}
    If the adjustment of $\left\{\big(\alpha_k,\theta_k,\beta_k\big)\right\}_k$ in Algorithm~\ref{algm:pdhg_sdp} follows that 
    \begin{align*}
        \alpha_k \in \left(\alpha_{\min}, \alpha_{\max}\right),\quad \sum_{k=1}^{\infty}\left|\alpha_{k+1}-\alpha_{k}\right|<\infty,\quad \theta_k = \frac{\alpha_k}{\alpha_{k-1}},\quad \alpha_k\beta_k = R,
    \end{align*}
    where $0<\alpha_{\min}\leq\alpha_{\max}<\infty$ and $R< \frac{1}{\lambda_{\max}(\mathcal{A}^T\mathcal{A})}$. Then Algorithm~\ref{algm:pdhg_sdp} weakly converges to $(X^*, y^*)$ such that $0 \in C + \partial_{X}\mathbb{I}_{\mathbb{S}_{+}^{n\times n}} (X^*)  +  \partial_{X}\mathbb{I}_{=b} (\mathcal{A}(X^*))$.
\end{restatable}

Theorem~\ref{theorem_pdhg4sdp} provides an unified approach to analyze the convergence of PDHG with adaptive stepsizes. We show that various existing variants, such as Algorithm~\ref{algm:pdhg_sdp_pdb}  and \ref{algm:pdhg_sdp_aa}, are special cases of Algorithm~\ref{algm:pdhg_sdp}, and thus justify their convergence (Theorem~\ref{theorem_pdhg4sdp_balancing}).

\begin{algorithm}[H]
\captionsetup{font=small}
\small
\begin{algorithmic}
\STATE \textbf{Input:} $X^0 \in \mathbb{R}^{n\times n}, y^0 \in \mathbb{R}^m, \alpha_0\beta_0 < \frac{1}{\lambda_{\max}(\mathcal{A}^T\mathcal{A})}, (\epsilon_0,\eta)\in(0,1)^2$
\FOR{$k = 0, 1, 2, \dots $}
\STATE $X^{k+1} = \text{Proj}_{\mathcal{S}^{n}_{+}}\Big(X^k-\alpha_k\big(\mathcal{A}^T(y^k)- C\big)\Big)$
\STATE $y^{k+1} = y^k + \beta_k \bigg(\mathcal{A}\Big(X^{k+1} + \theta_{k}\big(X^{k+1} - X^k\big)\Big) - b\bigg)$
\STATE $p^{k+1} = \norm{\frac{X^k-X^{k+1}}{\alpha_k}-\mathcal{A}^T(y^k-y^{k+1})}$
\STATE $d^{k+1} = \norm{\frac{y^k-y^{k+1}}{\beta_k} - \mathcal{A}\big(X^k-X^{k+1}\big)}$
\IF{$p^{k+1} > 2 d^{k+1} \Delta$}
    \STATE $\alpha_{k+1} = \frac{\alpha_k}{1-\epsilon_k},\ \beta_{k+1} = \beta_{k}(1-\epsilon_k),\ \theta_{k+1} = \frac{1}{1-\epsilon_{k}}$
  \ELSIF{$\frac{1}{2} d^{k+1} \leq p^{k+1} \leq 2 d^{k+1}$}
    \STATE $\alpha_{k+1} = \alpha_k,\ \beta_{k+1} = \beta_{k},\ \theta_{k+1} = 1$
  \ELSE
    \STATE $\alpha_{k+1} = \alpha_k(1-\epsilon_k),\ \beta_{k+1} = \frac{\beta_{k}}{1-\epsilon_k},\ \theta_{k+1} = 1-\epsilon_{k}$
  \ENDIF
\STATE $\epsilon_{k+1} = \epsilon_k \eta$
\ENDFOR
\end{algorithmic}
\caption{\small Balancing Primal and Dual Residuals (B-PDR) \cite{goldstein2013adaptive,goldstein2015adaptive}}
\label{algm:pdhg_sdp_pdb}
\end{algorithm}

\begin{algorithm}[H]
\captionsetup{font=small}
\small
\begin{algorithmic}
\STATE \textbf{Input:} $X^0 \in \mathbb{R}^{n\times n}, y^0 \in \mathbb{R}^m, \alpha_0\beta_0 < \frac{1}{\lambda_{\max}(\mathcal{A}^T\mathcal{A})}, (\epsilon_0,\eta)\in(0,1)^2$
\FOR{$k = 0, 1, 2, \dots $}
\STATE $X^{k+1} = \text{Proj}_{\mathcal{S}^{n}_{+}}\Big(X^k-\alpha_k\big(\mathcal{A}^T(y^k)- C\big)\Big)$
\STATE $y^{k+1} = y^k + \beta_k \bigg(\mathcal{A}\Big(X^{k+1} + \theta_{k}\big(X^{k+1} - X^k\big)\Big) - b\bigg)$
\STATE $p^{k+1} = \frac{1}{\alpha_k}(X^k-X^{k+1})-\mathcal{A}^T(y^k-y^{k+1})$ or $\frac{1}{\alpha_k}(X^k-X^{k+1})-\mathcal{A}^T (y^k)$
\STATE $w_{p}^{k+1} = \frac{\langle X^k - X^{k+1}, p^{k+1}\rangle}{\norm{X^{k}-X^{k+1}}_{\mathcal{F}}\norm{p^{k+1}}_{\mathcal{F}}}$ 
\IF{$w^{k+1}_{p} > 0.99 $}
    \STATE $\alpha_{k+1} = \frac{\alpha_k}{1-\epsilon_k},\ \beta_{k+1} = \beta_{k}(1-\epsilon_k),\ \theta_{k+1} = \frac{1}{1-\epsilon_{k}}$
  \ELSIF{$0 \leq w^{k+1}_{p} \leq 0.99$}
    \STATE $\alpha_{k+1} = \alpha_k,\ \beta_{k+1} = \beta_{k},\ \theta_{k+1} = 1$
  \ELSE
    \STATE $\alpha_{k+1} = \alpha_k(1-\epsilon_k),\ \beta_{k+1} = \frac{\beta_{k}}{1-\epsilon_k},\ \theta_{k+1} = 1-\epsilon_{k}$
  \ENDIF
\STATE $\epsilon_{k+1} = \epsilon_k \eta$
\ENDFOR
\end{algorithmic}
\caption{\small Aligning Local Variation of Gradient Directions (A-LV) \cite{yokota2017efficient,zdun2021fast}}
\label{algm:pdhg_sdp_aa}
\end{algorithm}

\begin{algorithm}[H]
\captionsetup{font=small}
\small
\begin{algorithmic}
\STATE \textbf{Input:} $X^0,y^1,\alpha_0>0,s\in\mathbb{R}$
\FOR{$k = 1, 2, \dots $}
\STATE $X^{k} \leftarrow \text{Proj}_{\mathbb{S}_{+}^{n\times n}}\Big(X^{k-1} - \alpha_{k-1} \big(\mathcal{A}^T(y^k)+C\big)\Big)$
\STATE $\alpha_{k} \in [\alpha_{k-1},\alpha_{k-1}\sqrt{1+\theta_{k-1}}]$, $\beta_k = s\alpha_k$ and $\theta_k \leftarrow \frac{\alpha_k}{\alpha_{k-1}}$
\STATE $y^{k+1} \leftarrow y^k + \beta_k\big( \mathcal{A}(X^{k}-\theta_k(X^{k}-X^{k-1})) - b\big)$
\WHILE{$\norm{\mathcal{A}^T(y^{k+1}) - \mathcal{A}^T(y^k)} > \frac{1}{\sqrt{s}\alpha_{k}}\norm{y^{k+1}-y^k}$}
\STATE $\alpha_k \leftarrow \alpha_k\mu$, $\beta_k = s\alpha_k$ and $\theta_k \leftarrow \frac{\alpha_k}{\alpha_{k-1}}$
\STATE $y^{k+1} \leftarrow y^k + \beta_k\big( \mathcal{A}(X^{k}-\theta_k(X^{k}-X^{k-1})) - b\big)$
\ENDWHILE
\ENDFOR
\end{algorithmic}
\caption{\small Linesearch (LS) \cite{malitsky2018first}}
\label{algm:pdhg_sdp_line_search}
\end{algorithm}

\begin{restatable}{thm-rest}{ThPDHGforSDPB}
\label{theorem_pdhg4sdp_balancing}
    Algorithm~\ref{algm:pdhg_sdp_pdb} and \ref{algm:pdhg_sdp_aa} converge to $(X^*, y^*)$ such that $0 \in C + \partial_{X}\mathbb{I}_{\mathbb{S}_{+}^{n\times n}} (X^*)  +  \partial_{X}\mathbb{I}_{=b} (\mathcal{A}(X^*))$.
\end{restatable}

It is worth noting that Algorithm~\ref{algm:pdhg_sdp_pdb} and \ref{algm:pdhg_sdp_aa}, as well as the linesearch strategy (Algorithm~\ref{algm:pdhg_sdp_line_search}), enable adaptive stepsize adjustment, but they are not tuning-free. They introduce some additional hyperparameters that  still need to be carefully tuned. As verified in our numerical experiments, the hyperparameters  that lead to
fast convergence in one scenario may not be a good choice for another. And manually tuning them is a tedious and disturbing process. To address this challenge, we develop a stepsize-adjusting strategy for PDHG (Algorithm~\ref{algm:pdhg_sdp_drs}) that is totally tuning-free. This strategy is originally designed for DRS \cite{lorenz2019non}. We adapt it to the PDHG scenario based on the connection \cite{o2020equivalence} between PDHG and DRS. We prove its convergence by showing that it is a special case of Algorithm~\ref{algm:pdhg_sdp} (Theorem~\ref{theorem_pdhg4sdp_drs}).

\begin{algorithm}[H]
\captionsetup{font=small}
\small
\begin{algorithmic}
\STATE \textbf{Input:} $X^0,y^1,\epsilon \geq \lambda_{\max}(\mathcal{A}^T\mathcal{A}),\theta_{\min} = 10^{-5}, \theta_{\max} = 10^{5}, \alpha_0 =1, \{\omega_k\}_{k } = \{2^{-\frac{k}{100}}\}_k$.
\FOR{$k = 1, 2, \dots $}
\STATE $X^{k} \leftarrow \text{Proj}_{\mathbb{S}_{+}^{n\times n}}\Big(X^{k-1} - \alpha_{k-1} \big(\mathcal{A}^T(y^k)+C\big)\Big)$
\STATE $\theta_k \leftarrow \text{Proj}_{[\theta_{\min},\theta_{\max}]}\Big(\frac{\norm{
        X^{k} }}{\norm{
        X^{k} - X^{k-1} + \alpha_{k-1}\mathcal{A}^T(y^k)}} \Big)$, $\alpha_k \leftarrow (1-\omega_k + \omega_k\theta_k)\alpha_{k-1}$, and $\beta_k \leftarrow \frac{1}{\epsilon\alpha_k}$
\STATE $y^{k+1} \leftarrow y^k + \beta_k \mathcal{A}(X^{k}+\theta_k(X^{k}-X^{k-1})) - \beta_k b$
\ENDFOR
\end{algorithmic}
\caption{\small Tuning-Free PDHG for SDP} 
\label{algm:pdhg_sdp_drs}
\end{algorithm}

\begin{restatable}{thm-rest}{ThPDHGforSDPDRS}
\label{theorem_pdhg4sdp_drs}
    Algorithm~\ref{algm:pdhg_sdp_drs} converges to $(X^*, y^*)$ such that $0 \in C + \partial_{X}\mathbb{I}_{\mathbb{S}_{+}^{n\times n}} (X^*)  +  \partial_{X}\mathbb{I}_{=b} (\mathcal{A}(X^*))$.
\end{restatable}

\section{Numerical Experiments}
\label{sec:exp}
We compare the performances of Algorithm~\ref{algm:pdhg_sdp_drs} with three existing adaptive alternatives: balancing primal dual residuals (B-PDR, Algorithm~\ref{algm:pdhg_sdp_pdb}) \cite{goldstein2013adaptive}, aligning local variation of gradient directions (A-LV, Algorithm~\ref{algm:pdhg_sdp_aa}) \cite{yokota2017efficient}, and linesearch (LS, Algorithm~\ref{algm:pdhg_sdp_line_search}) \cite{malitsky2018first}. We test them on: 
\begin{enumerate}
    \item \textbf{Random generation (RG)}: $C,\mathcal{A},b$ of \eqref{prob:primal_sdp} is randomly generated, while feasibility is ensured. In the simulations, we set $m = 50$ and $n = 50$.
    \item \textbf{Maximum cut (MC)}: Given an undircted graph $G = (V,E)$ comprising a vertex set $V = \{1,\dots,n\}$ and a set $E$ of $m$ edges, an SDP relaxation for solving maximum cut is
    \begin{align*}
        \min_{x} &\quad \langle L, X\rangle \\
        \text{s.t.} &\quad 
                    \text{diag}(X) = 1 \\
                    &\quad X \succeq 0.
    \end{align*}
    where $L \coloneqq \sum_{\{i,j\}\in E}(e_i-e_j)(e_i-e_j)^T \in \mathbb{R}^{n\times n}$. In the simulations, we set $m = 100$ and $n = 100$.
    \item \textbf{Sensor network localization (SNL)}:
Given anchors $\{a_1,a_2,\dots,a_{m}\} \subset \mathbb{R}^p$, distance information $d_{ij},\ (i,j)\in \mathcal{N}_x \coloneqq \Big\{(i,j): i<j,\ d_{ij}\ \text{is specified}\Big\}$, and $\hat{d}_{kj}, \ (k,j)\in \mathcal{N}_a \coloneqq \Big\{(k,j): \hat{d}_{kj}\ \text{is specified}\Big\}$, sensor network localization finds $\{x_{1},x_{2},\dots,x_{n}\} \subset \mathbb{R}^p$ for all $i$ such that
\begin{align*}
    & \norm{x_i-x_j}^2 = d_{ij}^2,\quad\forall\ (i,j)\in \mathcal{N}_x, \\
    & \norm{a_k-x_j}^2 = \hat{d}_{kj}^2,\quad\forall\ (k,j)\in \mathcal{N}_a.
\end{align*}
Let $X = \big[x_1\ x_2\ \cdots\ x_n\big]\in\mathbb{R}^{p\times n}$. Then an SDP relaxation for solving this problem is finding a symmetric matrix $Z\in \mathbb{S}^{(p+n)\times (p+n)}$ such that
\begin{align*}
    \min_{Z} &\quad 0 \\
    \text{s.t.} &\quad 
                (e_i-e_j)(e_i-e_j)^T\bullet Y = d_{ij}^2,\ \forall (i,j) \in \mathcal{N}_x\\
                &\quad \begin{pmatrix}
                    a_k \\
                    -e_j
                \end{pmatrix}{\begin{pmatrix}
                    a_k \\
                    -e_j
                \end{pmatrix}}^T\bullet \begin{pmatrix}
                    I_{p \times p} & X \\
                    X^T & Y 
                \end{pmatrix} = \hat{d}_{kj}^2,\ \forall (k,j) \in \mathcal{N}_a \\
                &\quad Z \coloneqq \begin{pmatrix}
                    I_{p \times p} & X \\
                    X^T & Y 
                \end{pmatrix} \succeq 0.
\end{align*}
In the simulations, we set $m=10, n = 50, \texttt{radius} = 0.3, \texttt{degree} = 5$. Different from random generation and maximum cut, here $m$ refers to the number of anchors and $n$ refers to the total number of sensors.
\end{enumerate}

\subsection{Settings and parameter configurations}
The computational bottleneck of applying PDHG to solve large-scale SDP is the need for full eigenvalue decomposition at each iteration when projecting onto the positive semidefinite cone. To address this, practitioners usually adopt a rounding heuristics, updating the primal variables $X^k$ by an approximate projection onto $\mathbb{S}_{+}^{n\times n}$:
$$
X^{k} \leftarrow \text{A-Proj}^{r}_{\mathbb{S}_{+}^{n\times n}}\Big(\underbrace{X^{k-1} - \alpha_{k-1} \big(\mathcal{A}^T(y^k)+C\big)}_{\hat{X}^{k-1}}\Big) = \sum_{i=1}^{r} \max\{0,\lambda_i\}u_i u_i^T \in \mathbb{S}_{+}^{n\times n},
$$
where $\lambda_1 \geq \lambda_2 \geq \dots \geq \lambda_r$ are the first $r$ eigenvalues of $\hat{X}^{k-1}$, and $u_1,u_2,\dots,u_r$ are the associated eigenvectors. This approximation can be efficiently computed by iterative algorithms, such as the Lanczos algorithm, which only requires matrix-vector products. We set $r$ to be $\log(n)$ in the simulations.

For B-PDR (Algorithm~\ref{algm:pdhg_sdp_pdb}) and A-LV (Algorithm~\ref{algm:pdhg_sdp_aa}), we set $\epsilon_0 = 0.5$ and perform a grid search on $\eta \in \left \{0.9,0.925,0.95,0.975,0.99\right\}$. We test each set of configurations on $33$ random generation, $34$ maximum cut, and $33$ sensor network localization with different random seeds, resulting in $100$ different structures of $C,\mathcal{A},b$ (See the code for details of the generation). $\eta$ is finally set to be $0.95$, as B-PDR and A-LV with such configuration solves $86\%$ and $83\%$ of the simulated problems with the fastest convergence speed. For LS (Algorithm~\ref{algm:pdhg_sdp_line_search}), we set $s \in \{\frac{1}{10},\frac{1}{5},1,10\}$.

For the stopping criteria, the algorithms are terminated once $\norm{p^k}^2 + \norm{d^k}^2 < {10}^{-6}$, where
\begin{align*}
    & p^{k+1} = \frac{1}{\alpha_k}(x^k-x^{k+1})-A^T(y^k-y^{k+1}) \in \partial f(x^{k+1}) + A^Ty^{k+1} = p(x^{k+1},y^{k+1}), \\
    & d^{k+1} = \frac{1}{\beta_k}(y^k-y^{k+1})-A(x^k-x^{k+1}) \in \partial g(y^{k+1}) + Ax^{k+1} = d(x^{k+1},y^{k+1}).
\end{align*}

\subsection{Results and discussion}
For each problem in \{RG, MC, SNL\}, we run $100$ times with different random seeds, resulting in $100$ different structures of $C,\mathcal{A},b$. Table~\ref{tab:random} records the percent of the solved problems within a certain number of iterations for each algorithm. For example, the entry $\left(2,2\right)$ in Table~\ref{tab:random} means the proposed tuning-free algorithm solves $38\%$ of the simulated RG problems within $5000$ iterations. Figure~\ref{fig:random} plots the $\left\{\log(\norm{p^k}^2 + \norm{d^k}^2) -\texttt{iteration}\right\}$-curve with respect to $\texttt{rng(1)}, \texttt{rng(2)}, \texttt{rng(3)}, \texttt{rng(4)}$. 

The proposed tuning-free algorithm shows superior convergence speed and robustness, outperforming other baselines. Moreover, linesearch does not have a one-size-fit-all choice of $s$. Indeed, its performances highly depends on problems and settings. For example, in the second row of Figure~\ref{fig:random}, linesearch with $s=\frac{1}{5}$ achieves the fastest convergence speed for $\texttt{rng(3)}$, achieves a mediocre convergence speed for $\texttt{rng(1)}$, and diverges in the rest.

\begin{figure}[!htb]
\minipage{0.25\textwidth}
  \includegraphics[width=\linewidth]{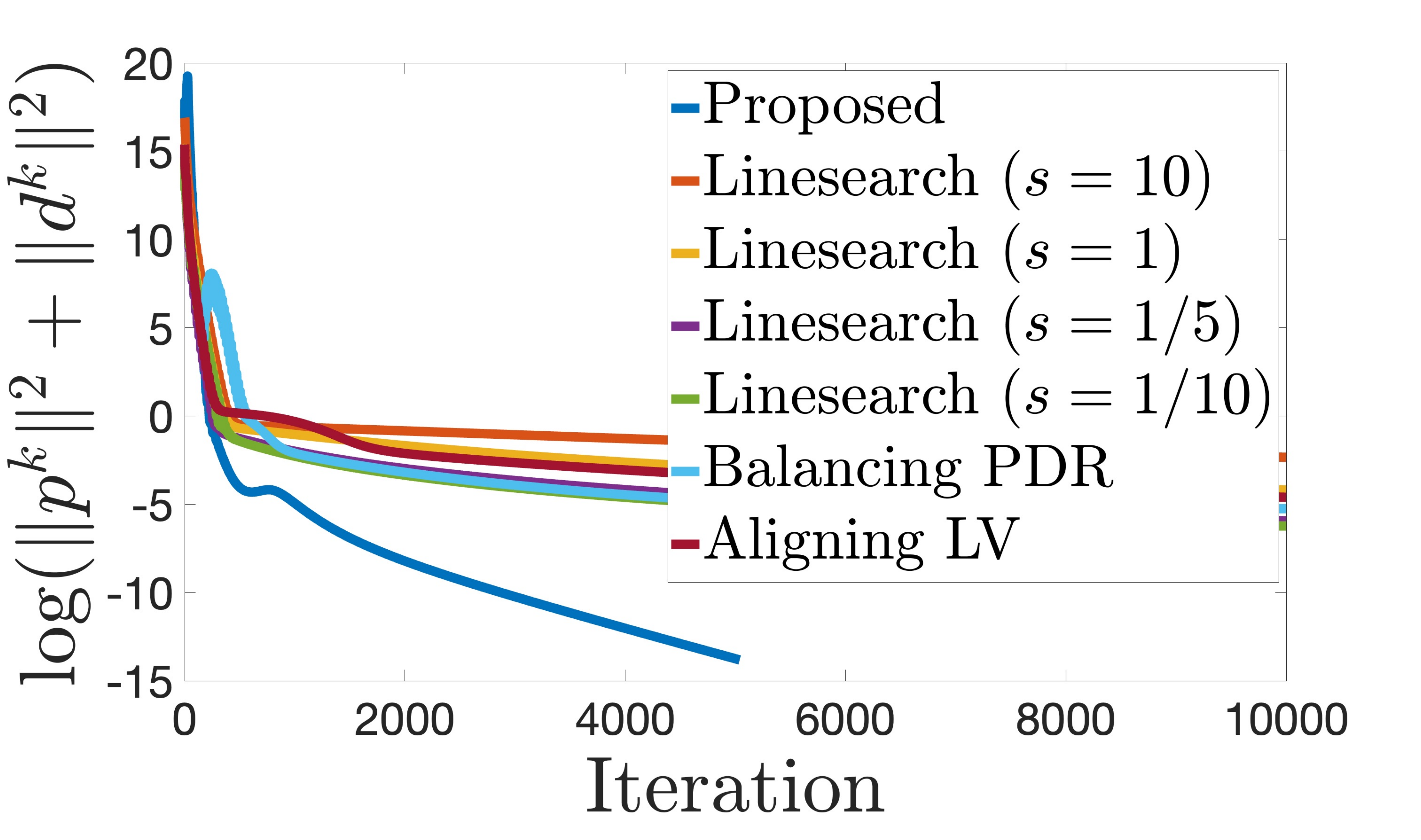}
\endminipage\hfill
\minipage{0.25\textwidth}
  \includegraphics[width=\linewidth]{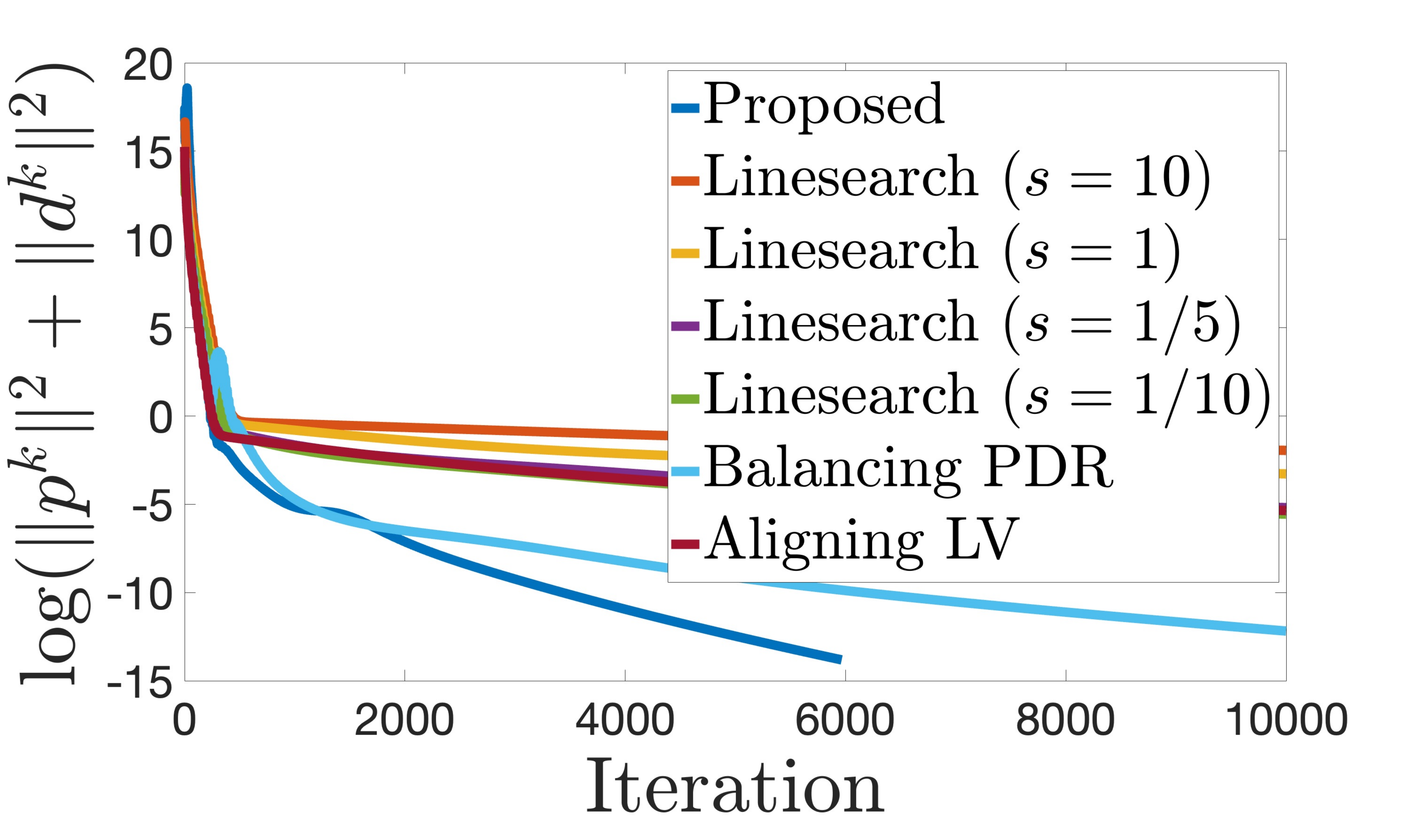}
\endminipage\hfill
\minipage{0.25\textwidth}
  \includegraphics[width=\linewidth]{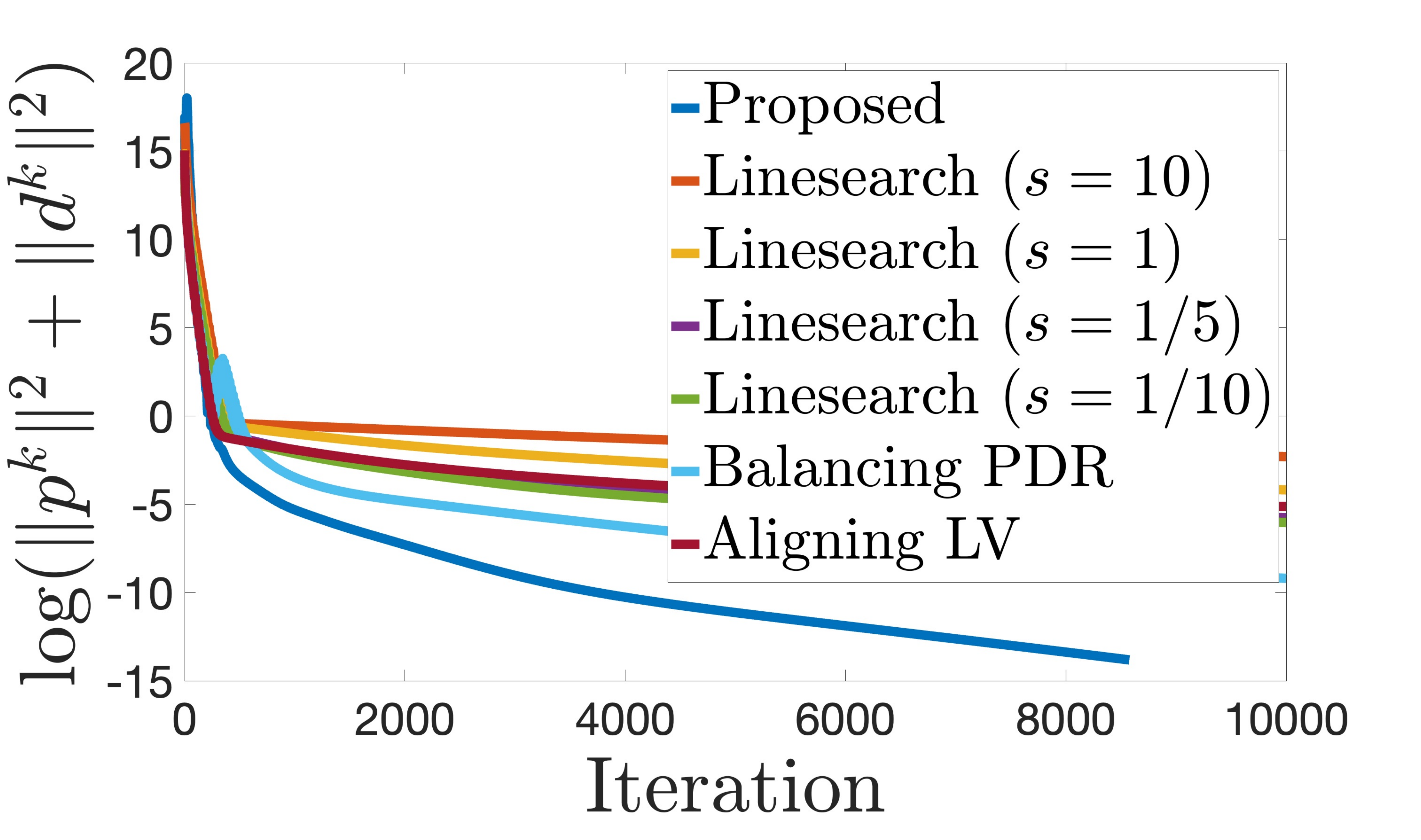}
\endminipage\hfill
\minipage{0.25\textwidth}
  \includegraphics[width=\linewidth]{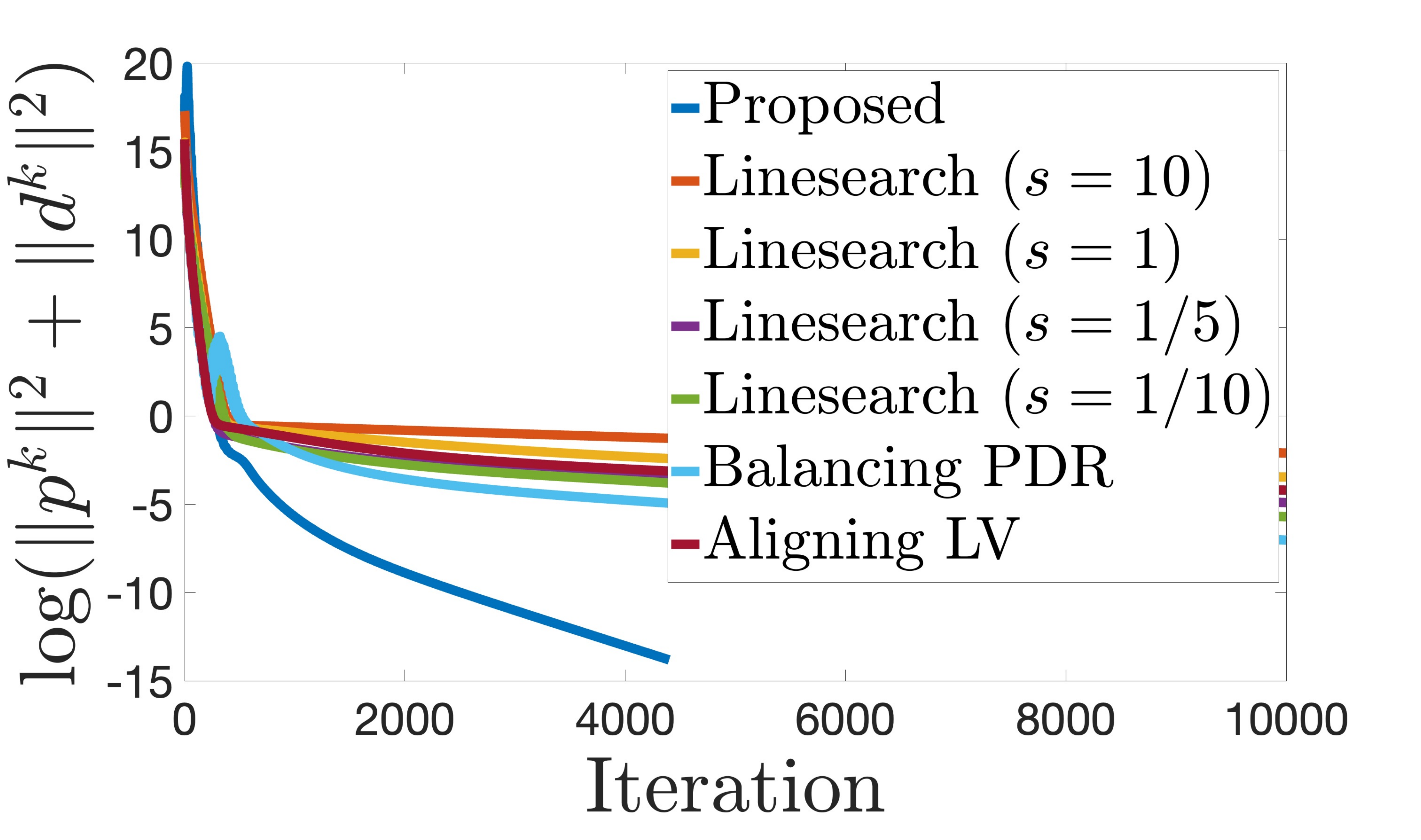}
\endminipage\hfill
\minipage{0.25\textwidth}
  \includegraphics[width=\linewidth]{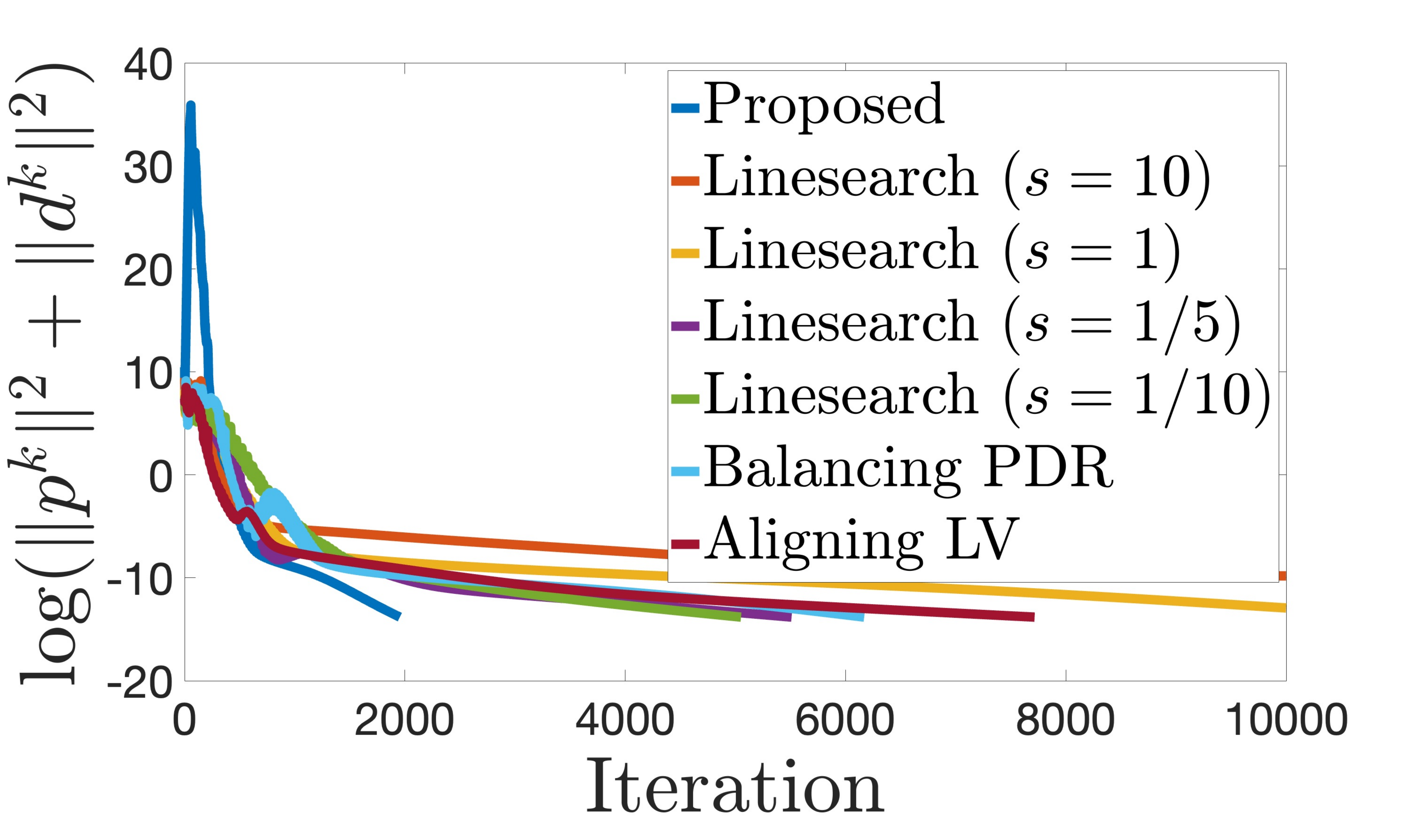}
\endminipage\hfill
\minipage{0.25\textwidth}
  \includegraphics[width=\linewidth]{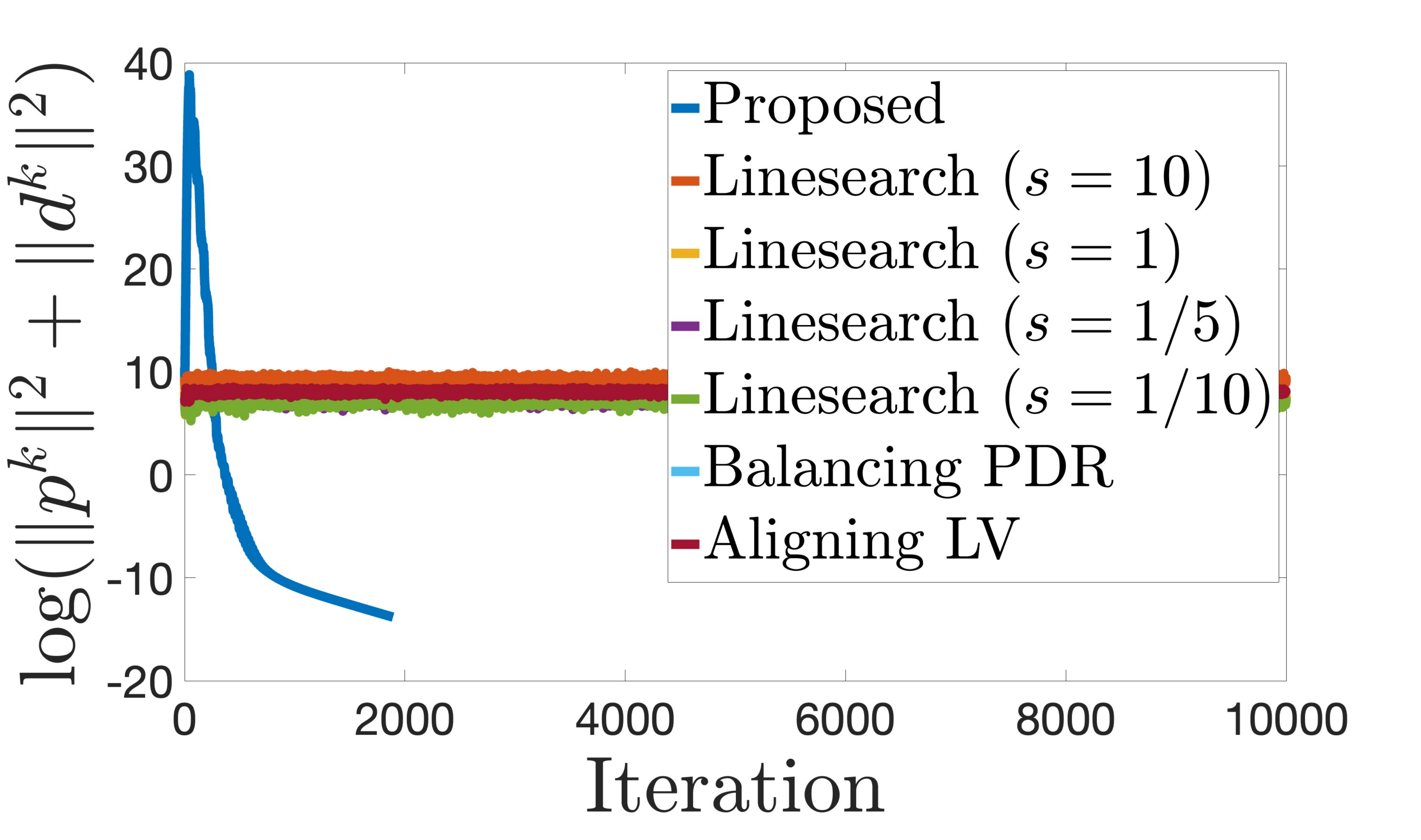}
\endminipage\hfill
\minipage{0.25\textwidth}
  \includegraphics[width=\linewidth]{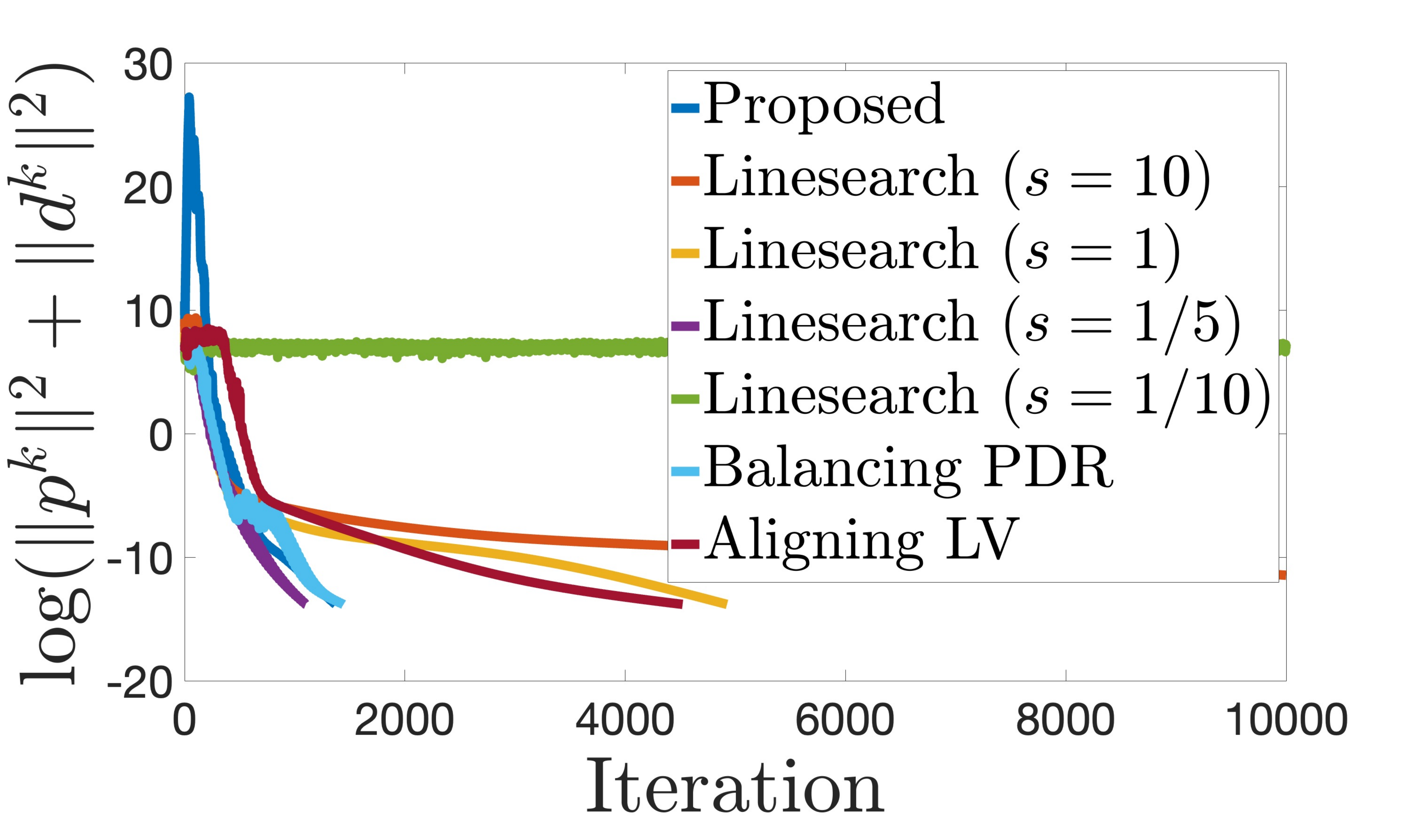}
\endminipage\hfill
\minipage{0.25\textwidth}
  \includegraphics[width=\linewidth]{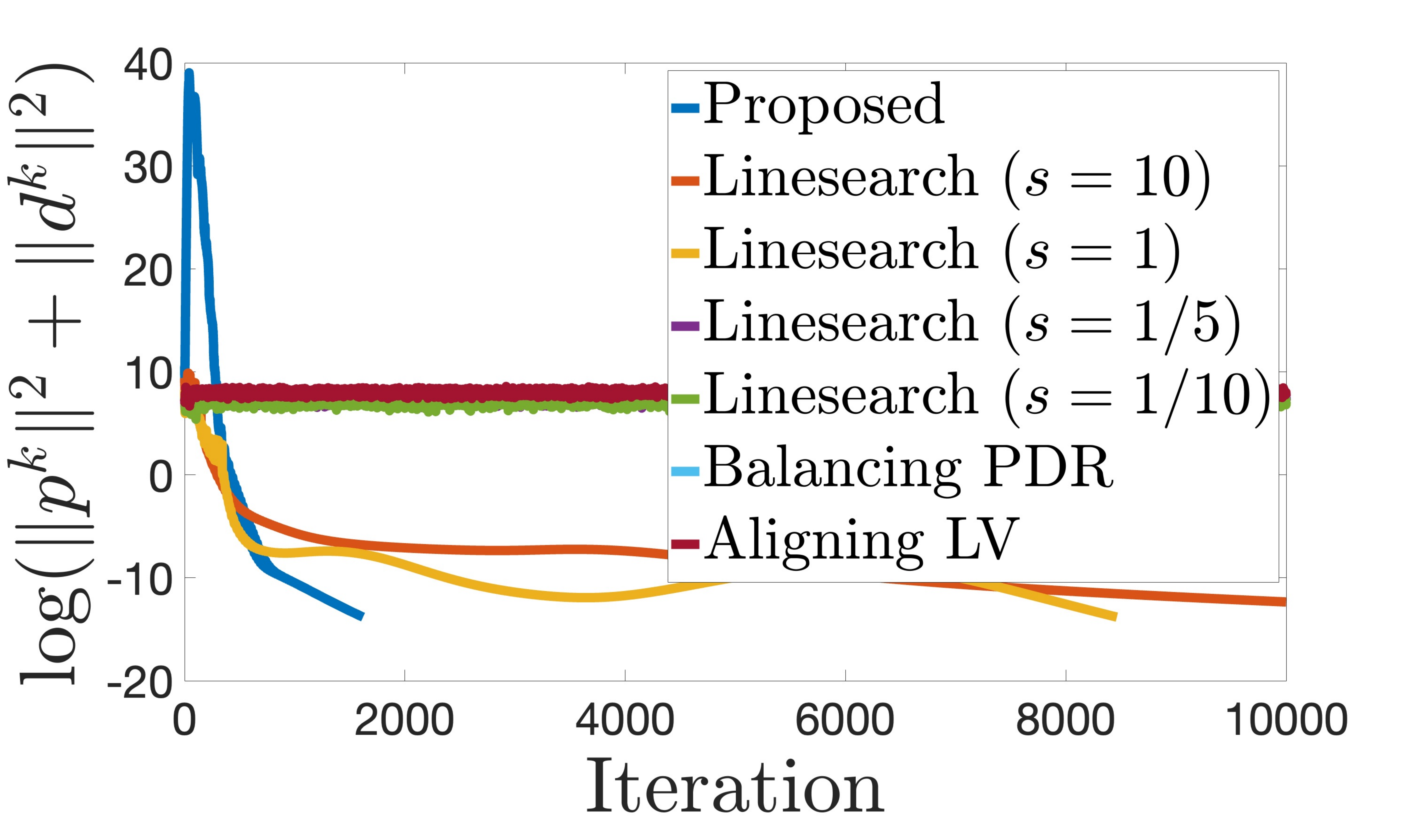}
\endminipage\hfill
\minipage{0.25\textwidth}
  \includegraphics[width=\linewidth]{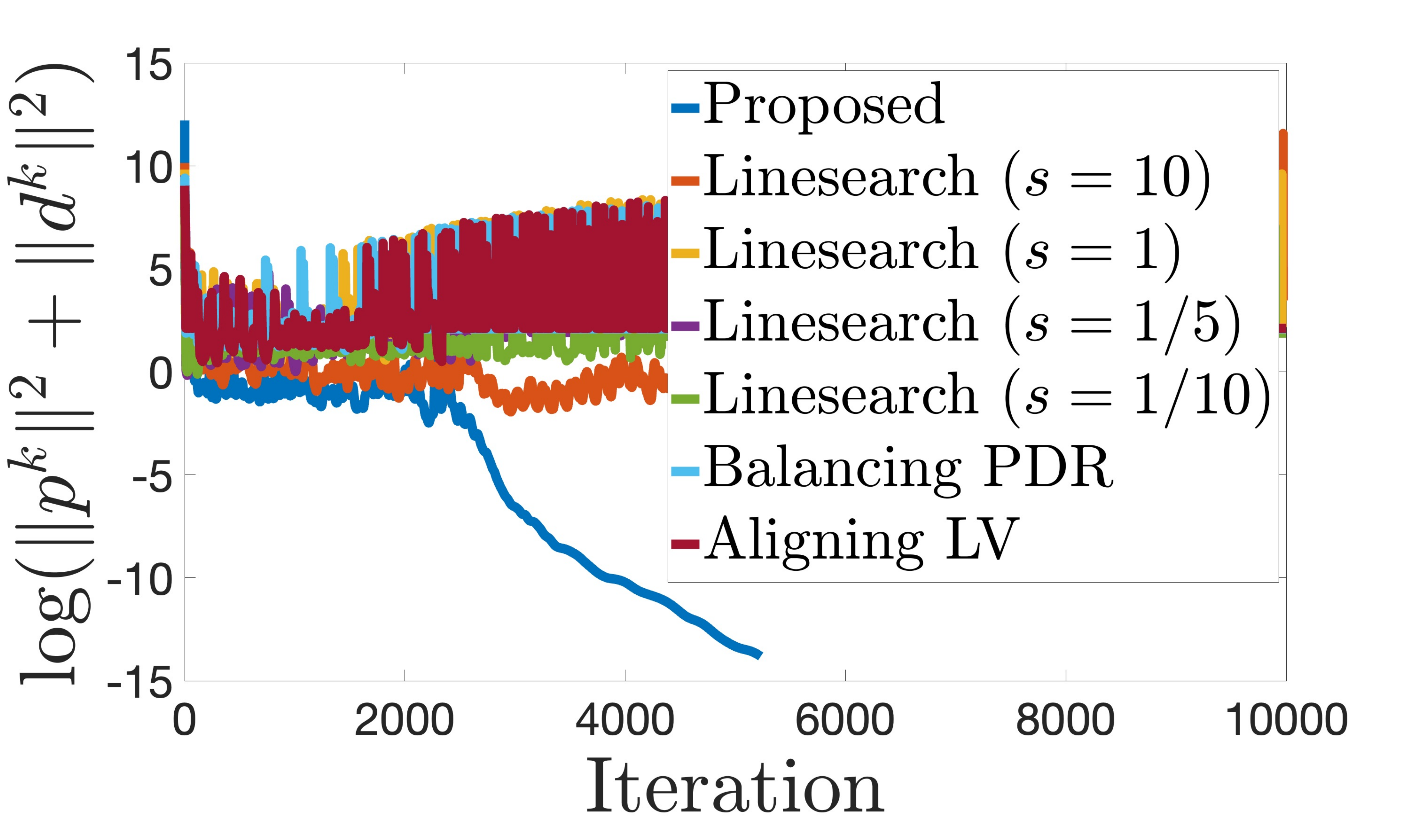}
\endminipage\hfill
\minipage{0.25\textwidth}
  \includegraphics[width=\linewidth]{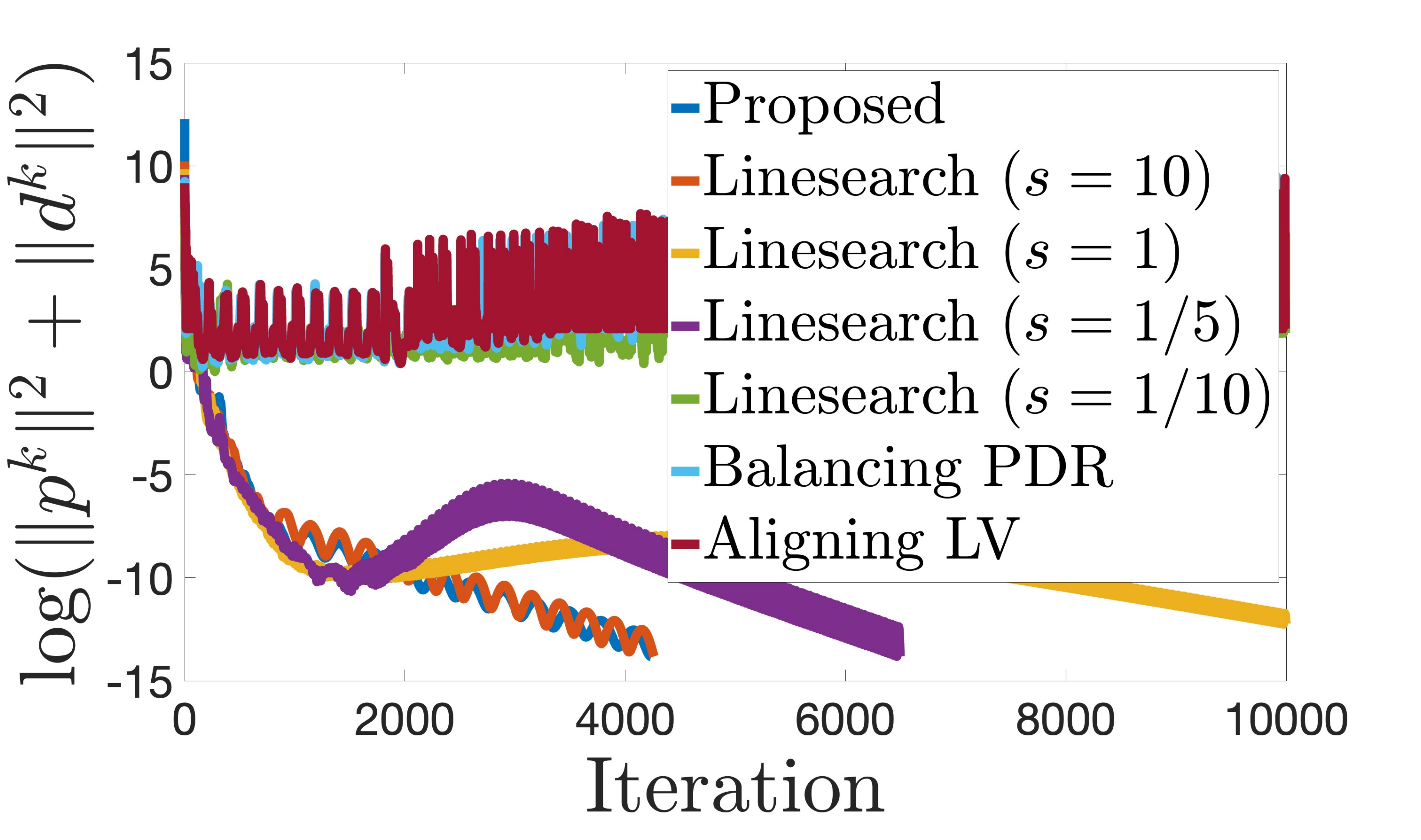}
\endminipage\hfill
\minipage{0.25\textwidth}
  \includegraphics[width=\linewidth]{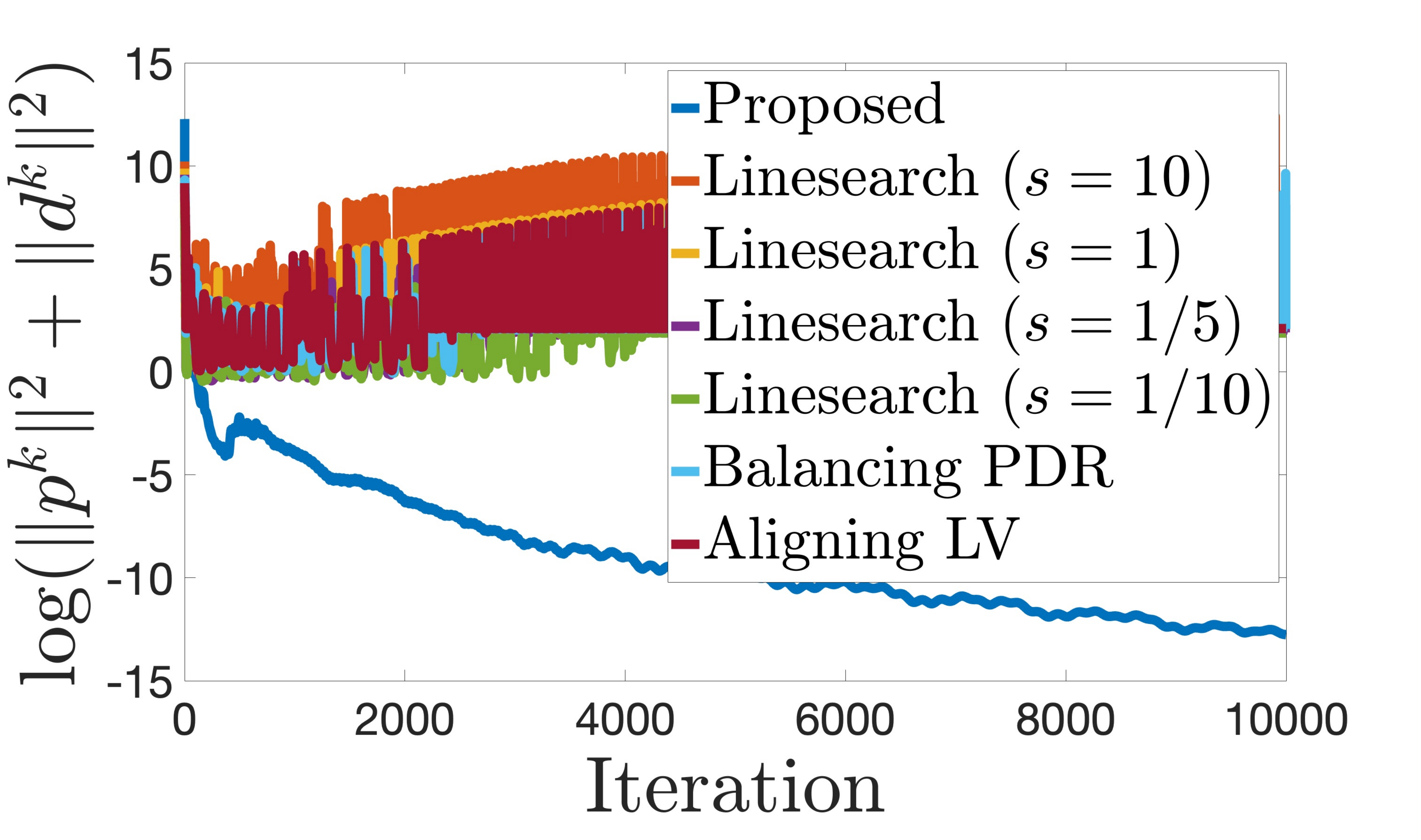}
\endminipage\hfill
\minipage{0.25\textwidth}
  \includegraphics[width=\linewidth]{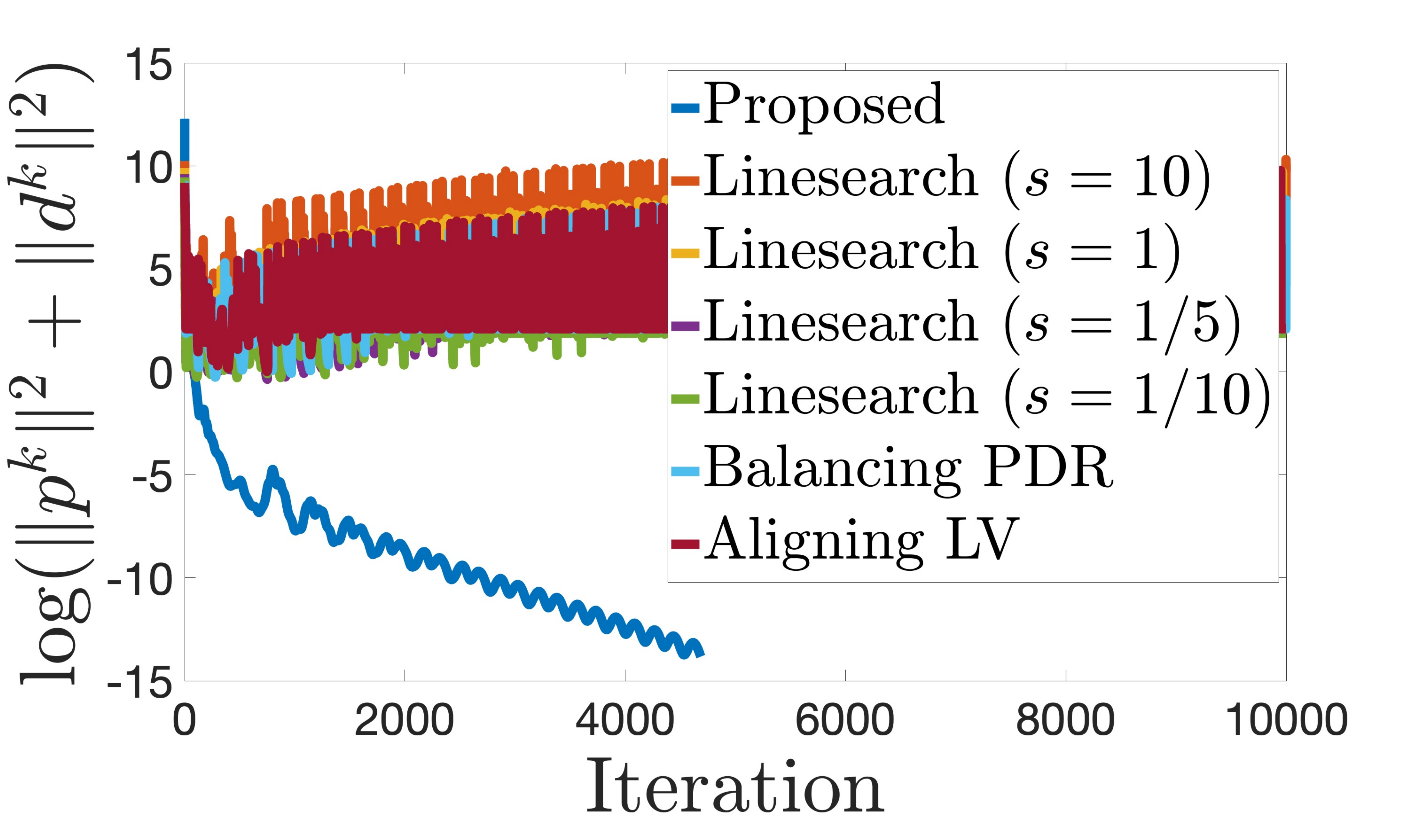}
\endminipage
\caption{The rows from top to bottom are the results of RG, MC, and SNL. The columns from left to right are the results of $\texttt{rng(1)}$, $\texttt{rng(2)}$, $\texttt{rng(3)}$, and $\texttt{rng(4)}$.}
\label{fig:random}
\end{figure}

\begin{table}[h!]
\small
\centering
\begin{tabular}{|c|c|c|c|c|c|c|c|}
\hline
\textbf{RG} & Proposed & LS $(s=10)$ & LS $(s=1)$ & LS $(s=1/5)$ & LS $(s=1/10)$ & B-PDR & A-LV \\
\hline
   $< 5000$  & $38 \%$ & $5 \%$  &   $0 \%$      &   $1 \%$    & $3 \%$   & $8 \%$ & $3 \%$  \\
\hline
  $< 10000$ & $55 \%$ & $21 \%$ & $22 \%$ & $17 \%$    & $18 \%$  & $34 \%$ & $41 \%$    \\
\hline
  $< 25000$ & $89 \%$ & $78 \%$  &    $62 \%$    &   $75 \%$    & $80 \%$    & $73 \%$ & $76 \%$      \\
\hline
 \textbf{MC} & Proposed & LS $(s=10)$ & LS $(s=1)$ & LS $(s=1/5)$ & LS $(s=1/10)$ & B-PDR & A-LV \\
\hline
   $< 2500$  & $70 \%$ & $0 \%$  &   $7 \%$      &   $23 \%$    & $20 \%$   & $45 \%$ & $33 \%$  \\
\hline
  $< 5000$ & $87 \%$ & $25 \%$ & $19 \%$ & $26 \%$    & $35 \%$  & $69 \%$ & $36 \%$    \\
\hline
  $< 10000$ & $91 \%$ & $63 \%$  &    $80 \%$    &   $73\%$    & $69 \%$    & $88 \%$ & $66 \%$      \\
\hline
 \textbf{SNL} & Proposed & LS $(s=10)$ & LS $(s=1)$ & LS $(s=1/5)$ & LS $(s=1/10)$ & B-PDR & A-LV \\
\hline
   $< 7500$  & $24 \%$ & $22 \%$  &   $0 \%$      &   $10 \%$    & $5 \%$   & $11 \%$ & $1 \%$  \\
\hline
  $< 15000$ & $60 \%$ & $30 \%$ & $27 \%$ & $25 \%$    & $16 \%$  & $34 \%$ & $41 \%$    \\
\hline
  $< 30000$ & $73 \%$ & $54 \%$  &    $33 \%$    &   $59 \%$    & $40 \%$    & $62 \%$ & $62 \%$      \\
\hline
\end{tabular}
\caption{The blocks from top to bottom are the results of RG, MC, and SNL.}
\label{tab:random}
\end{table}

\newpage
\appendix
\section{Appendix}
\subsection{Proof of Theorem~\ref{theorem_pdhg4sdp}}
\ThPDHGforSDP*
\begin{proof}
The core idea of the proof comes from Section 4 of \cite{o2020equivalence}, which shows that PDHG can be viewed as a special case of DRS. It allows us to translate known convergence results from \cite{lorenz2019non} on non-stationary DRS to Algorithm~\ref{algm:pdhg_sdp}. 

We first show that Algorithm~\ref{algm:pdhg_sdp} is a instance of non-stationary DRS, which is extensively studied in \cite{lorenz2019non},
\begin{align*}
    z^{k+1} = J_{\alpha_{k}A}\left(J_{\alpha_{k-1}B}(z^{k}) + \frac{\alpha_{k}}{\alpha_{k-1}}\Big(J_{\alpha_{k-1}B}(z^{k}) - z^{k}\Big)\right) + 
    \frac{\alpha_k}{\alpha_{k-1}}\left(z^{k} - J_{\alpha_{k-1}B}(z^{k})\right),
\end{align*}
where $A$ and $B$ are two maximally monotone operators. Note that if the stepsize is fixed, non-stationary DRS is reduced to the vanilla one
\begin{align*}
    z^{k+1} = J_{\alpha A}\left(2 J_{\alpha B}(z^{k}) - z^{k}\right) + 
    z^{k} - J_{\alpha B}(z^{k}).
\end{align*}
Similar to \cite{o2020equivalence}, we introduce a linear operator $\mathcal{T}\colon \mathbb{S}^{n\times n} \to \mathbb{R}^m$ and then reformulate the primal form of SDP \eqref{prob:primal_sdp} as 
\begin{align*}
\min_{X} \quad  \langle C, X \rangle + \mathbb{I}_{\mathbb{S}_{+}^{n\times n}} (X) + \mathbb{I}_{=0} (\hat{X}) +  \mathbb{I}_{=b} \left(\mathcal{A}(X) + \mathcal{T}(\hat{X})\right),
\end{align*}
which is equivalent to the following monotone inclusion problem:
\begin{align}
\label{prob:sdp_mip}
    \text{Find}\ \begin{pmatrix}
        X \\
        \hat{X}
    \end{pmatrix}\quad \text{s.t.}\quad  0 \in  \underbrace{\begin{pmatrix} C \\
    0 \end{pmatrix} + \begin{pmatrix}\partial_{X}\mathbb{I}_{\mathbb{S}_{+}^{n\times n}} (X) \\
    0
    \end{pmatrix} + 
    \begin{pmatrix} 0 \\
    \partial_{\hat{X}}\mathbb{I}_{=0} (\hat{X})
    \end{pmatrix}}_{B = \partial f(X,\hat{X})} + \underbrace{\begin{pmatrix}  \partial_{X}\mathbb{I}_{=b} (\mathcal{A}(X) + \mathcal{T}(\hat{X})) \\ 
    \partial_{\hat{X}}\mathbb{I}_{=b} (\mathcal{A}(X) + \mathcal{T}(\hat{X}))
    \end{pmatrix}}_{A = \partial g(X,\hat{X})},
\end{align}
where $\partial f(X,\hat{X})$ and $\partial g(X,\hat{X})$ are maximally monotone. Applying non-stationary DRS to \eqref{prob:sdp_mip} with $z = \begin{pmatrix}
        X \\
        \hat{X}
    \end{pmatrix}$, we get the following fixed-point iteration:
\begin{align}
    & \begin{pmatrix}
        X^{k} \\
        \hat{X}^{k}
    \end{pmatrix} 
    =
    J_{\alpha_{k-1}\partial f}\Bigg(\begin{pmatrix}
        Z^{k} \\
        \hat{Z}^{k}
    \end{pmatrix}\Bigg) \label{update_X12}\\
    & \begin{pmatrix}
        X^{k+\frac{1}{2}} \\
        \hat{X}^{k+\frac{1}{2}}
    \end{pmatrix} 
    = J_{\alpha_{k}\partial g}\Bigg(\begin{pmatrix}
        X^{k} \\
        \hat{X}^{k}
    \end{pmatrix} + \theta_k\bigg(\begin{pmatrix}
        X^{k} \\
        \hat{X}^{k}
    \end{pmatrix} - \begin{pmatrix}
        Z^{k} \\
        \hat{Z}^{k}
    \end{pmatrix}\bigg)\Bigg) \label{update_X1}\\
    & \begin{pmatrix}
        Z^{k+1} \\
        \hat{Z}^{k+1}
    \end{pmatrix} 
    = \begin{pmatrix}
        X^{k+\frac{1}{2}} \\
        \hat{X}^{k+\frac{1}{2}}
    \end{pmatrix} + \theta_k\Bigg(\begin{pmatrix}
        Z^{k} \\
        \hat{Z}^{k}
    \end{pmatrix} - \begin{pmatrix}
        X^{k} \\
        \hat{X}^{k}
    \end{pmatrix}\Bigg)\label{update_Z}.
\end{align}
\eqref{update_X12} can be simplified to
\begin{align}
    X^{k} &= \argmin_{X}\Big\{\langle C, X \rangle + \mathbb{I}_{\mathbb{S}_{+}^{n\times n}}(X)+\frac{1}{2\alpha_{k-1}}\norm{X-Z^k}^2\Big\} \label{update_X12_new}\\
    \hat{X}^{k} &= 0 \notag
\end{align}
\eqref{update_X1} can be simplified to 
\begin{align}
    \begin{pmatrix}
        X^{k+\frac{1}{2}} \\
        \hat{X}^{k+\frac{1}{2}}
    \end{pmatrix} 
    & = J_{\alpha_{k}\partial g}\Bigg(\begin{pmatrix}
        X^{k} \\
        \hat{X}^{k}
    \end{pmatrix} + \theta_k\bigg(\begin{pmatrix}
        X^{k} \\
        \hat{X}^{k}
    \end{pmatrix} - \begin{pmatrix}
        Z^{k} \\
        \hat{Z}^{k}
    \end{pmatrix}\bigg)\Bigg) \nn\\
    & = \Bigg(1+\alpha_{k}\partial g\Bigg)^{-1}\Bigg(\begin{pmatrix}
        X^{k} \\
        \hat{X}^{k}
    \end{pmatrix} + \theta_k\bigg(\begin{pmatrix}
        X^{k} \\
        \hat{X}^{k}
    \end{pmatrix} - \begin{pmatrix}
        Z^{k} \\
        \hat{Z}^{k}
    \end{pmatrix}\bigg)\Bigg) \nn\\
    & \Longleftrightarrow \nn \\
    & \begin{pmatrix}
        X^{k}+\theta_k(X^{k}-Z^k) \\
        -\theta_k\hat{Z}^{k}
    \end{pmatrix} \in \begin{pmatrix}
        X^{k+\frac{1}{2}} \\
        \hat{X}^{k+\frac{1}{2}}
    \end{pmatrix} + \alpha_k\mathcal{A}^T\bigg(\partial \mathbb{I}_{=b}\Big(\mathcal{A}\big(X^{k+\frac{1}{2}}) + \mathcal{T}\big(\hat{X}^{k+\frac{1}{2}})\Big)\bigg) \nn \\
    & \Longleftrightarrow \nn \\
    & y^{k+1} \in \partial \mathbb{I}_{=b}\Big(\mathcal{A}\big(X^{k+\frac{1}{2}}) + \mathcal{T}\big(\hat{X}^{k+\frac{1}{2}})\Big) \nn\\
    & X^{k+\frac{1}{2}} + \alpha_k \mathcal{A}^T(y^{k+1}) = X^{k}+\theta_k(X^{k}-Z^k) \nn\\
    & \hat{X}^{k+\frac{1}{2}} + \alpha_k \mathcal{T}^T(y^{k+1}) = -\theta_k\hat{Z}^{k}   \nn\\
    & \Longleftrightarrow \nn \\
    & \mathcal{A}\big(X^{k+\frac{1}{2}}) + \mathcal{T}\big(\hat{X}^{k+\frac{1}{2}}) \in \partial_{y} \mathbb{I}^{*}_{=b}(y^{k+1}) \nn\\
    & X^{k+\frac{1}{2}} = X^{k}+\theta_k(X^{k}-Z^k) - \alpha_k \mathcal{A}^T(y^{k+1}) \nn\\
    & \hat{X}^{k+\frac{1}{2}} = -\theta_k\hat{Z}^{k} - \alpha_k \mathcal{T}^T(y^{k+1}) \nn\\
    & \Longleftrightarrow \nn \\
    & \mathcal{A}\big(X^{k}+\theta_k(X^{k}-Z^k)\big)-\theta_k \mathcal{T}\big(\hat{Z}^k\big)-\alpha_k \mathcal{A}\big(\mathcal{A}^T(y^{k+1})\big) -\alpha_k \mathcal{T}\big(\mathcal{T}^T(y^{k+1})\big) \in \partial_{y} \mathbb{I}^{*}_{=b}(y^{k+1}) \nn\\
    & X^{k+\frac{1}{2}} = X^{k}+\theta_k(X^{k}-Z^k) - \alpha_k \mathcal{A}^T(y^{k+1}) \nn\\
    & \hat{X}^{k+\frac{1}{2}} = -\theta_k\hat{Z}^{k} - \alpha_k \mathcal{T}^T(y^{k+1}) \nn\\
    & \Longleftrightarrow \nn \\
    & y^{k+1} = \argmin_{y} \Big\{\mathbb{I}_{=b}^{*} (y)-\Big\langle\mathcal{A}\big(X^{k}+\theta_k(X^{k}-Z^k)\big)-\theta_k \mathcal{T}\big(\hat{Z}^k\big), y\Big\rangle \nn\\
    & \quad\quad\quad\quad\quad+ \frac{\alpha_k}{2}(\norm{\mathcal{A}^T(y)}^2+\norm{\mathcal{T}^T(y)}^2)\Big\} \label{update_u}\\
    & X^{k+\frac{1}{2}} = X^{k}+\theta_k(X^{k}-Z^k) - \alpha_k \mathcal{A}^T(y^{k+1}) \nn\\
    & \hat{X}^{k+\frac{1}{2}} = -\theta_k\hat{Z}^{k} - \alpha_k \mathcal{T}^T(y^{k+1}) \nn,
\end{align}
\eqref{update_Z} can be simplified to 
\begin{align}
    Z^{k+1} &= X^{k} -\alpha_k \mathcal{A}^T(y^{k+1}) \label{update_Z_new}\\
    \hat{Z}^{k+1} &=  - \alpha_k \mathcal{T}^T(y^{k+1}). \label{update_hatZ_new}
\end{align}
Substituting \eqref{update_Z_new} into \eqref{update_X12_new}, we get
\begin{align}
X^{k} &= \argmin_{X}\Big\{\langle C, X \rangle + \mathbb{I}_{\mathbb{S}_{+}^{n\times n}}(X)+\frac{1}{2\alpha_{k-1}}\norm{X-\big(X^{k-1}  -\alpha_{k-1} \mathcal{A}^T(y^k)\big)}^2\Big\} \notag\\
&= \text{Prox}_{\alpha_{k-1} \left(\langle C, \cdot \rangle + \mathbb{I}_{\mathbb{S}_{+}^{n\times n}} (\cdot)\right)}\Big(X^{k-1}  -\alpha_{k-1} \mathcal{A}^T(y^k)\Big) \notag\\
&= \text{Prox}_{\alpha_{k-1}  \mathbb{I}_{\mathbb{S}_{+}^{n\times n}} (\cdot)}\Big(X^{k-1} - \alpha_{k-1} \mathcal{A}^T(y^k)-\alpha_{k-1} C\Big)\notag\\
&= \text{Proj}_{\mathbb{S}_{+}^{n\times n}}\Big(X^{k-1} - \alpha_{k-1} \mathcal{A}^T(y^k)-\alpha_{k-1} C\Big)
\end{align}
Substituting \eqref{update_Z_new} and \eqref{update_hatZ_new} into \eqref{update_u}, we get
\begin{align*}
    y^{k+1} &= \argmin_{y} \Big\{\mathbb{I}_{=b}^{*} (y)-\langle\mathcal{A}(X^{k}+\theta_k(X^{k}-X^{k-1})), y\rangle + \frac{\alpha_k}{2}\norm{y -y^{k}}^2_{\mathcal{A}\mathcal{A}^T+\mathcal{T}\mathcal{T}^T}\Big\}.
\end{align*}
By Lemma~\ref{lemma1}, we can find a $\mathcal{T}$ such that $\mathcal{T}\mathcal{T}^T = \frac{1}{R}\mathcal{I}-\mathcal{A}\mathcal{A}^T$, $\alpha_k\beta_k = R < \frac{1}{\lambda_{\max}(\mathcal{A}^T\mathcal{A})}$ for $k=1,2,\dots$, and then get
\begin{align*}
    y^{k+1} &= \text{Prox}_{\beta_k \left(\mathbb{I}_{=b} (\cdot)\right)^*}(y^k + \beta_k \mathcal{A}(X^{k}-\theta_k(X^{k}-X^{k-1}))) \\
    &= y^k + \beta_k \mathcal{A}(X^{k+1}-\theta_k(X^{k}-X^{k-1})) - \beta_k \text{Prox}_{\frac{1}{\beta_k} \mathbb{I}_{=b} (\cdot)}(\frac{y^k + \beta_k \mathcal{A}(X^{k}-\theta_k(X^{k}-X^{k-1}))}{\beta_k}) \\
    &= y^k + \beta_k \mathcal{A}(X^{k}-\theta_k(X^{k}-X^{k-1})) - \beta_k \text{Proj}_{=b}(\frac{y^k + \beta_k \mathcal{A}(X^{k}-\theta_k(X^{k}-X^{k-1}))}{\beta_k}) \\
    &= y^k + \beta_k \mathcal{A}(X^{k}-\theta_k(X^{k}-X^{k-1})) - \beta_k b.
\end{align*}
Finally, we have
\begin{align*}
    & X^{k} = \text{Proj}_{\mathbb{S}_{+}^{n\times n}}\Big(X^{k-1} - \alpha_{k-1} \mathcal{A}^T(y^k)-\alpha_{k-1} C\Big) \\
    & y^{k+1} = y^k + \beta_k \mathcal{A}(X^{k}-\theta_k(X^{k}-X^{k-1})) - \beta_k b
\end{align*}
where $X^{k}$ and $y^{k+1}$ are the primal and dual variables of SDP, respectively. Since Algorithm~\ref{algm:pdhg_sdp} is an instance of non-stationary DRS, we can conclude the proof by Theorem~\ref{theorem_drs}.
\end{proof}

\begin{lemma}
\label{lemma1}
    There exists linear operator $\mathcal{T}\colon \mathbb{R}^{n\times n} \rightarrow \mathbb{R}^{m}$ such that $\mathcal{T}\mathcal{T}^T = \frac{1}{R}\mathcal{I}-\mathcal{A}\mathcal{A}^T$.
\end{lemma}
\begin{proof}
    A constructive proof is provided here. Let the vectorization of $X,A_1,\dots,A_m$ be $\text{vec}(X),\text{vec}(A_1),\dots,\text{vec}(A_m)\in\mathbb{R}^{n^2}$, respectively. Define the matrix representation of linear operators $\mathcal{T}$ and $\mathcal{A}$ as 
    \begin{align*}
        T = \begin{bmatrix}
            \text{vec}(T_1)^T \\
            \vdots \\
            \text{vec}(T_m)^T
    \end{bmatrix}\in\mathbb{R}^{m\times n^2},\quad 
        A = \begin{bmatrix}
            \text{vec}(A_1)^T \\
            \vdots \\
            \text{vec}(A_m)^T
    \end{bmatrix}\in\mathbb{R}^{m\times n^2}.
    \end{align*}
    Define $S:=\frac{1}{R} I-AA^T\succ 0$. Since $m\leq n^2$, $T$ can be constructed as
    \begin{align*}
        T = \begin{bmatrix}
            \sqrt{\lambda_1}v_1 \\
            \vdots \\
            \sqrt{\lambda_m}v_m \\
            \mathbf{0}^{m\times (n^2-m)}
        \end{bmatrix},
    \end{align*}
    where $S$ is eigenvalue-decomposed as $S=\sum_{i=1}^{m}\lambda_i v_iv_i^T$. Since $TT^T=S$, $T$ is the desired matrix representation of $\mathcal{T}$. Therefore, we have explicitly construct the linear operator $\mathcal{T}$.
\end{proof}

\begin{restatable}{thm-rest}{ThDRS}
\label{theorem_drs}
    Let $A$ and $B$ be maximally monotone and $\{\alpha_k\}_k$ be a positive sequence such as
    \begin{align*}
        \alpha_k \in \left(\alpha_{\min}, \alpha_{\max}\right),\quad \sum_{k=1}^{\infty}\left|\alpha_{k+1}-\alpha_{k}\right|<\infty,
    \end{align*}
    where $0<\alpha_{\min}\leq \alpha_{\max} < \infty$. Then non-stationary DRM
    \begin{align*}
    z^{k+1} = J_{\alpha_{k}A}\left(J_{\alpha_{k-1}B}(z^{k}) + \frac{\alpha_{k}}{\alpha_{k-1}}\Big(J_{\alpha_{k-1}B}(z^{k}) - z^{k}\Big)\right) + 
    \frac{\alpha_k}{\alpha_{k-1}}\left(z^{k} - J_{\alpha_{k-1}B}(z^{k})\right),
    \end{align*}
    weakly converges to $z^*$ such that $0\in (A+B)(z^*)$.
\end{restatable}
\begin{proof}
    It is the restatement of Theorem~3.2. of \cite{lorenz2019non}.
\end{proof}

\subsection{Proof of Theorem~\ref{theorem_pdhg4sdp_balancing}}
\ThPDHGforSDPB*
\begin{proof}
    Since $\theta_k = \frac{\alpha_k}{\alpha_{k-1}}$ and $\alpha_k\beta_k = R < \frac{1}{\lambda_{\max}\left(\mathcal{A}^T\mathcal{A}\right)}$ are automatically satisfied by Algorithm~\ref{algm:pdhg_sdp_pdb} and \ref{algm:pdhg_sdp_aa}, it suffices to shows $\alpha_k \in \left(\alpha_{\min}, \alpha_{\max}\right)$ and $\sum_{k=1}^{\infty}\left|\alpha_{k+1}-\alpha_{k}\right|<\infty$. To prove the boundedness of $\alpha_k$, we assume $\alpha_k$ is increased at each iteration without loss of generality, i.e., $\alpha_{k+1} = \frac{\alpha_k}{1-\epsilon_k} = \alpha_0\prod_{i=0}^{k+1} \frac{1}{1-\epsilon_i} = \alpha_0\prod_{i=0}^{k+1} \frac{1}{1-\epsilon_0  \eta^{i}}$. We wish to show $\left\{ \alpha_k \right\}$ is a convergent sequence, so that it is bounded. And it is sufficient to show $\left\{x_k\right\}$, where $x_{k+1} = \log \left(\prod_{i=0}^{k+1}\frac{1}{1-\epsilon_i}\right) = \sum_{i=0}^{k+1}\log\left(\frac{1}{1-\epsilon_0  \eta^{i}}\right)$, is convergent since $\alpha_0$ is a constant. Since $\lim _ {x \to 0} \frac{\ln (1+x)}{x} = 1$, we have
    \begin{align*}
        \lim_{i \to \infty} \frac{\log \left(\frac{1}{1-\epsilon_0 \eta^{i+1}}\right)}{\log \left( \frac{1}{1-\epsilon_0  \eta^{i}} \right)} = \lim_{i \to \infty} \frac{\log \left(\frac{1}{1-\epsilon_0  \eta^{i+1}}\right)}{\frac{\epsilon_0\eta^{i+1} }{1-\epsilon_0\eta^{i+1}}}\times \frac{\frac{\epsilon_0\eta^{i+1} }{1-\epsilon_0\eta^{i+1}}}{\frac{\epsilon_0\eta^{i} }{1-\epsilon_0\eta^{i}}} \times \frac{\frac{\epsilon_0\eta^{i} }{1-\epsilon_0\eta^{i}}}{\log \left(\frac{1}{1-\epsilon_0  \eta^{i}}\right)} = \lim_{i \to \infty} \frac{\eta - \eta^{i+1} \epsilon_0}{1-\eta^{i+1} \epsilon_0} =\eta < 1.
    \end{align*}
By ratio test, $x_n$ is convergent. We now show that $\sum_{k=1}^{\infty}\left|\alpha_{k+1}-\alpha_{k}\right|<\infty$. Since 
    $1-\epsilon_0\eta^k > \frac{1}{2} \Leftrightarrow \log\left(\frac{1}{2\epsilon_0}\right) > k\log\left(\eta\right) \Leftrightarrow \underbrace{\frac{\log\left(\frac{1}{2\epsilon_0}\right)}{\log\left(\eta\right)}}_{=k^*} > k$, 
    we have
    \begin{align*}
        \sum_{k=0}^{\infty} \left|\alpha_{k+1}-\alpha_k\right| & \leq \sum_{k=0}^{\infty}\left|\frac{\alpha_k}{1-\epsilon_k} - \alpha_k\right| + \sum_{k=0}^{\infty}\left|(1-\epsilon_k)\alpha_k-\alpha_k\right| + \sum_{k=0}^{\infty}\left|\alpha_k - \alpha_k\right| \\
        & \leq \sum_{k=0}^{\infty}\frac{\epsilon_k}{1-\epsilon_k}\alpha_k + \sum_{k=0}^{\infty}\epsilon_k\alpha_k \\
        & \leq \alpha_{\max}\sum_{k=0}^{\infty}\left(\frac{\epsilon_k}{1-\epsilon_k}+\epsilon_k\right) \\
        & = \alpha_{\max}\sum_{k=0}^{\infty}\left(\frac{\epsilon_0\eta^k}{1-\epsilon_0\eta^k}+\epsilon_0\eta^k\right) \\
        & \leq \alpha_{\max}\sum_{k=0}^{\left\lceil k^* \right\rceil}\left(\frac{\epsilon_0\eta^k}{1-\epsilon_0\eta^k}+\epsilon_0\eta^k\right) + \alpha_{\max}\sum_{k=\left\lceil k^* \right\rceil+1}^{\infty}\left(\frac{\epsilon_0\eta^k}{\frac{1}{2}}+\epsilon_0\eta^k\right) \\
        &= \alpha_{\max}\sum_{k=0}^{\left\lceil k^* \right\rceil}\left(\frac{\epsilon_0\eta^k}{1-\epsilon_0\eta^k}+\epsilon_0\eta^k\right) + 3\epsilon_0\alpha_{\max}\sum_{k=\left\lceil k^* \right\rceil+1}^{\infty}\eta^k \\
        & = \alpha_{\max}\sum_{k=0}^{\left\lceil k^* \right\rceil}\left(\frac{\epsilon_0\eta^k}{1-\epsilon_0\eta^k}+\epsilon_0\eta^k\right) + 3\epsilon_0\alpha_{\max}\frac{\eta^{\left\lceil k^* \right\rceil+1}}{1-\eta} < \infty.
    \end{align*}
We then conclude the proof by Theorem~\ref{theorem_drs}.    
\end{proof}

\subsection{Proof of Theorem~\ref{theorem_pdhg4sdp_drs}}
\ThPDHGforSDPDRS*
\begin{proof}
    Since $\langle C, X \rangle + \mathbb{I}_{\mathbb{S}_{+}^{n\times n}} (X) + \mathbb{I}_{=0} (\hat{X})$ and $\mathbb{I}_{=b} (\mathcal{A}(X) + \mathcal{T}(\hat{X}))$ are closed, convex, proper functions with respect to $\begin{pmatrix}
        X \\
        \hat{X}
    \end{pmatrix}$, we have
    \begin{align*}
        \begin{pmatrix} C \\
        0 \end{pmatrix} + \begin{pmatrix}\partial_{X}\mathbb{I}_{\mathbb{S}_{+}^{n\times n}} (X) \\
        0
        \end{pmatrix} + 
        \begin{pmatrix} 0 \\
        \partial_{\hat{X}}\mathbb{I}_{=0}     (\hat{X})
        \end{pmatrix}
    \end{align*}
    and
    \begin{align*}
        \begin{pmatrix}  \partial_{X}\mathbb{I}_{=b} (\mathcal{A}(X) + \mathcal{T}(\hat{X})) \\ 
        \partial_{\hat{X}}\mathbb{I}_{=b} (\mathcal{A}(X) + \mathcal{T}(\hat{X}))
        \end{pmatrix}
    \end{align*}
    are maximally monotone operators.  
    Since
    \begin{align*}
        \omega_i\min\{|1-\theta_{\text{min}}|, |1-\theta_{\text{max}}|\} \leq |(1-\omega_i+\omega_i\theta_i)-1| = \omega_i|1-\theta_i| \leq \omega_i\max\{|1-\theta_{\text{min}}|, |1-\theta_{\text{max}}|\},
    \end{align*}
    $\omega_i > 0,\ \forall i$ and $\sum_{i=0}^{\infty}\omega_i < \infty$, we have $\sum_{i=0}^{\infty}|(1-\omega_i+\omega_i\theta_i)-1| < \infty$. It implies $(1-\omega_i+\omega_i\theta_i)\rightarrow 1$ and $\prod_{i=1}^{k}(1-\omega_i+\omega_i\theta_i)\rightarrow c$ for a certain $c\in \mathbb{R}_{++}$.
    Therefore, we have 
    \begin{align*}
        \alpha_k = (1-\omega_k+\omega_k\theta_k)\alpha_{k-1} = \prod_{i=1}^{k} (1-\omega_i+\omega_i\theta_i)\alpha_0 \rightarrow c \alpha_0 > 0 \Longrightarrow \sum_{i=1}^{\infty}|\alpha_{k}-\alpha_{k-1}| < \infty.
    \end{align*}
    Moreover, since 
    \begin{align*}
        \theta_{\text{min}} = \min\{1,\theta_{\text{min}}\} \leq (1-\omega_i+\omega_i\theta_{\text{min}}) \leq (1-\omega_i+\omega_i\theta_i) \leq (1-\omega_i+\omega_i\theta_{\text{max}}) \leq \max\{1,\theta_{\text{max}}\} = \theta_{\text{max}},
    \end{align*}
    we have
    \begin{align*}
        \theta_{\text{min}}\alpha_0 \leq \alpha_k = \prod_{i=1}^{k} (1-\omega_i+\omega_i\theta_i)\alpha_0 \leq \theta_{\text{max}}\alpha_0,
        \end{align*}
    which implies $\alpha_k$ is uniformly bounded. We then conclude the proof by Theorem~\ref{theorem_drs}.
\end{proof}


\printbibliography
\end{document}